\documentclass[10pt,a4paper,reqno]{amsart}

\usepackage{lmodern}
\usepackage[T1]{fontenc}
\usepackage[utf8]{inputenc}
\usepackage[english]{babel}
\usepackage{microtype} 

\usepackage{slashed,url,bm,amssymb,mathrsfs,mathtools,todonotes,xparse,booktabs,graphicx}
\usepackage[centertableaux]{ytableau}
\usepackage[pdftex]{hyperref}
\usepackage{float,placeins}

\usepackage{fancyhdr}
\usepackage[font={small,sl}]{caption}
\newtheorem{theorem}{Theorem}[section]
\newtheorem{lemma}[theorem]{Lemma}
\newtheorem{definition}[theorem]{Definition}

\newtheorem{corollary}[theorem]{Corollary}
\newtheorem{conjecture}[theorem]{Conjecture}
\newtheorem{example}[theorem]{Example}

\definecolor{pBlue}{RGB}{86,139,190}
\definecolor{pOrange}{RGB}{211,153,80}
\definecolor{boxL}{RGB}{153,211,80}
\definecolor{boxU}{RGB}{80,153,211}

\newcommand{\defin}[1]{%
\relax\ifmmode%
\textcolor{blue}{#1}%
\else\textcolor{blue}{\emph{#1}}%
\fi%
}

\DeclareMathOperator{\Rect}{Rect}
\DeclareMathOperator{\Res}{Res}
\DeclareMathOperator{\arm}{arm}
\DeclareMathOperator{\leg}{leg}
\DeclareMathOperator*{\ord}{ord}

\newcommand{\setZ}{\mathbb{Z}}
\newcommand{\setR}{\mathbb{R}}
\newcommand{\setQ}{\mathbb{Q}}
\newcommand{\powerSum}{p}
\newcommand{\monomial}{m}
\newcommand{\xvec}{\mathbf{x}}

\newcommand{\hvec}{\mathbf{h}}
\newcommand{\jackj}{J}
\newcommand{\jacks}[1]{{J^{\#}_{#1}}} 
\newcommand{\jacksp}[1]{{P^{\#}_{#1}}} 

\newcommand{\jacktop}{\varpi}
\newcommand{\jackjLR}{g}

\newcommand{\addset}{\mathcal{O}}
\newcommand{\remset}{\mathcal{I}}
\newcommand{\bsigma}{\bar{\sigma}}
\newcommand{\bmu}{\bar{\mu}}

\newcommand{\partition}{\,\vdash\,}

\newcommand{\norm}[1]{\lVert#1\rVert} 
\renewcommand{\*}{\mathbin{\star}}

\newcommand{\uhv}[2]{{\hvec^{\scriptscriptstyle{\mathrm U}}_{#1}(#2)}}
\newcommand{\lhv}[2]{{\hvec^{\scriptscriptstyle{\mathrm L}}_{#1}(#2)}}
\newcommand{\uh}[2]{{h^{\scriptscriptstyle{\mathrm U}}_{#1}(#2)}}
\newcommand{\ulh}[2]{{h^{\scriptscriptstyle{\mathrm U/L}}_{#1}(#2)}}
\newcommand{\lh}[2]{{h^{\scriptscriptstyle{\mathrm L}}_{#1}(#2)}}

\newcommand{\eh}[2]{{h^{\scriptscriptstyle{\mathrm U/L}}_{#1}(#2)}}

\newcommand{\hflip}[1]{{ \check {#1} }}
\newcommand{\vflip}[1]{{ \hat {#1} }}

\newcommand{\bmin}[2]{{#1 \mathbin{\curlywedge} #2}}
\newcommand{\bmax}[2]{{#1 \mathbin{\curlyvee} #2}}

\DeclareMathSymbol{\shortminus}{\mathbin}{AMSa}{"39}
\def\Plus{\normalfont{\texttt{+}}}
\def\Minus{\normalfont{\texttt{-}}}
\newcommand{\bms}[1]{{#1^{{\Minus}}}}
\newcommand{\bps}[1]{{#1^{{\Plus}}}}

\newcommand{\bs}{\bar{s}}
\newcommand{\bt}{\bar{t}}
\newcommand{\bb}{\bar{b}}
\newcommand{\bw}{\bar{w}}

\newcommand{\lbox}{{1}} 

\title[The Stanley conjecture]{New cases of the Strong Stanley conjecture}

\author[P. Alexandersson]{Per Alexandersson}
\address{Department of Mathematics, Stockholm University, SE-106 91 Stockholm, Sweden}
\email{per.w.alexandersson@gmail.com}

\author[R. Mickler]{Ryan Mickler}
\address{Singulariti Research, 55 University Street, Carlton, 3053, Australia}
\email{ry.mickler@gmail.com}

\begin{document}
\maketitle

\begin{abstract}
We make progress towards understanding the structure of Littlewood--Richardson
coefficients $\jackjLR_{\lambda,\mu}^{\nu}$ for products of Jack symmetric functions.
Building on recent results of the second author, we are able to prove new cases of a
conjecture of Stanley in which certain families of these coefficients can
be expressed as a product of upper or lower hook lengths for every box in
each of the partitions. In particular, we prove that conjecture in the case
of a rectangular union, i.e. for $\jackjLR_{\mu,\bsigma}^{\mu \cup m^n}$
where $\bsigma$ is the complementary partition of $\sigma = \mu \cap m^n$
in the rectangular partition $m^n$. We give a formula for these coefficients
through an explicit prescription of such choices of hooks.
Lastly, we conjecture an analogue of this conjecture of Stanley holds
in the case of shifted Jack functions.
\end{abstract}


\section{Preliminaries}

\subsection{Young diagrams, hooks and the star product}

Given an integer partition $\lambda = (\lambda_1,\dotsc,\lambda_\ell)$, we 
associate to it the \defin{Young diagram} as the set of coordinates $(i,j)$ where $0\leq j < \ell$ and 
$0 \leq i \leq \lambda_{j-1}$ (French notation).
Each such coordinate $(i,j)$ is called a \defin{box}, and is pictorially
illustrated by the square with the four corners $\{(i,j),(i+1,j),(i,j+1),(i+1,j+1)\}$, see \eqref{eq:yDiagram}.
The bottom left hand corner of the diagram has coordinate $(0,0)$.
Given any box $s=(x,y)$, we let $\defin{\bps{s}}$ and $\defin{\bms{s}}$ denote the boxes (coordinate) $(x+1,y+1)$
and $(x-1,y-1)$, respectively.
In the diagram below, we have the marked the box $s=(2,1)$.
\begin{equation}\label{eq:yDiagram}
\ytableausetup{boxsize=0.9em}
 (7,4,2,2,1) \qquad \longleftrightarrow \qquad 
 \begin{ytableau}
\, \\ 
\, & \, \\ 
\, & \, \\ 
\, & \, & s & \, \\ 
\, & \, & \, & \, & \, & \, & \, \\ 
\end{ytableau}
\end{equation}
We let $\lambda$ also denote the set of boxes in the Young diagram, and $\defin{\lambda^\times}$ denotes 
the set of boxes excluding $(0,0)$.
Given a box $s \in \lambda$, let $\defin{\arm(s)}$ be the number of boxes to the right of $s$,
and let $\defin{\leg(s)}$ be the number of boxes above $s$.
For example, in Figure~\ref{fig:armLeg}, we have $\arm(s)=5$, $\leg(s)=3$.

We let $\defin{\addset_\lambda}$ be the set of boxes (called \defin{outer corners})
which can be added to $\lambda$ and still have a Young diagram
(denoted $\defin{\lambda \cup s}$), see Figure~\ref{fig:armLeg}.
Similarly, we let $\defin{\remset_\lambda}$ be the set of boxes (called \defin{inner corners})
which can be removed from $\lambda$ and still have a Young diagram.
For such $s \in \remset_\lambda$, let $\defin{\lambda \setminus s}$
denote the partition (diagram) obtained by removing $s$ from $\lambda$.
We shall also make use of
$\defin{\remset^+_\lambda} \coloneqq \{ \bps{s} : s \in \remset_\lambda\}$.
This set of coordinates has a natural interpretation, see Figure~\ref{fig:armLeg}.

Given two boxes $a=(a_1,a_2)$ and $b=(b_1,b_2)$
we let the \defin{join} \defin{$\bmin{a}{b}$} be defined as the box at
$(\min(a_1,b_1),\min(a_2,b_2))$. Similarly,
we define the \defin{meet} $\defin{\bmax{a}{b}}\coloneqq (\max(a_1,b_1),\max(a_2,b_2))$,
see Figure~\ref{fig:minmax}.

\begin{figure}[!ht]
\begin{equation*}
\centering
\includegraphics[scale=1.0,page=1]{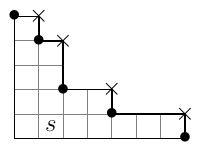}
\hspace{1cm}
\includegraphics[scale=1.0,page=1]{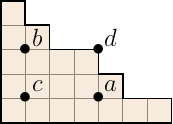}
\end{equation*}
\caption{Left: Consider $\lambda = 74221$ and $s=(1,0)$. For this box,
we have $\arm(s)=5$, $\leg(s)=3$.
All coordinates belonging to $\addset_\lambda$
have been marked with $\bullet$
and all elements in $\remset^+_\lambda$
are marked $\times$.
Right: $c= \bmin{a}{b}$ and $d = \bmax{a}{b}$.
}\label{fig:armLeg}\label{fig:minmax}
\end{figure}

We let $\defin{m^n}$ denote the rectangular partition $(m,m,\dotsc,m)$ with $n$ entries equal to $n$.
For fixed $m$, $n$ and $\mu \subseteq m^n$, we define the \defin{complement partition of $\mu$} as
\[
\defin{\bar{\mu}} \coloneqq (m-\mu_n,m-\mu_{n-1},\dotsc,m-\mu_2,m-\mu_1).
\]

For example, with $m=7$, $n=6$ and $\mu = 75421$ we have $\bar{\mu} = 76532$, as 
seen below:
\begin{equation*}\label{eq:complement}
\centering
\includegraphics[scale=1.0,page=1]{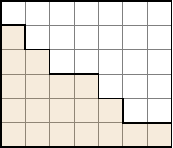}
\end{equation*}
Observe that $\addset_\mu$ mostly corresponds to $\remset^+_{\bar{\mu}}$, 
except for some boundary cases.

\medskip 

We let $\alpha$ be a fixed indeterminate and introduce the following notation.
For $s = (x,y) \in \setZ^2$, we set $\defin{[(x,y)]} \coloneqq  (x,y)\cdot(\alpha, -1) = \alpha x-y$.
We often simplify the notation and write $[x,y]$ instead of $[(x,y)]$.
Given $s \in \lambda$, we define the \defin{upper hook vector} and \defin{lower hook vector} as 
\begin{align*}
\defin{\uhv{\lambda}{s}} & \coloneqq (\arm(s)+1,-\leg(s)) \\
\defin{\lhv{\lambda}{s}} & \coloneqq (\arm(s), -\leg(s)-1),
\end{align*}
where we note that 
\begin{equation}\label{eq:lrhookshift}
\lhv{\lambda}{s} = \uhv{\lambda}{s} - (1,1),
\end{equation} and then the \defin{upper hook value} and \defin{lower hook value} as
\begin{align*}
\defin{\uh{\lambda}{s}} & \coloneqq [\uhv{\lambda}{s}] = \alpha (\arm(s)+1) + \leg(s), 
&& \text{(also known as $h^*_\lambda(s)$)}\\
\defin{\lh{\lambda}{s}} & \coloneqq [\lhv{\lambda}{s}] = \alpha \arm(s) + (\leg(s)+1)
&& \text{(also known as $h_*^\lambda(s)$).}
\end{align*}
We note that the difference is whether the contribution of the box at $s$ should count horizontally,
or vertically---see Figure~\ref{fig:upperLowerHooks} for an illustration.
\begin{figure}[!htb]
\centering
\includegraphics[scale=1.0,page=1]{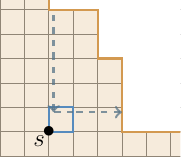}
\qquad 
\includegraphics[scale=1.0,page=2]{upperLowerHooks}
\caption{In the left picture, we have that $\uh{\lambda}{s} = 3\alpha +4$, while
$\lh{\lambda}{s} = 2\alpha + 5$.
}\label{fig:upperLowerHooks}
\end{figure}

\begin{definition}[Row and column sets of boxes]
\begin{equation*}\label{eq:rowcolSets}
\ytableausetup{boxsize=0.9em,aligntableaux=center}
\substack{
\begin{ytableau}
\, \\ 
\, & \, \\ 
*(pBlue) \, & *(pBlue)  \, & *(pBlue)  \,  & *(pBlue)  s \\ 
\, & \, & \, & \,  & \, \\ 
\, & \, & \, & \, & \, & \, & \, \\ 
\end{ytableau}
\\
R^s_\lambda
}
\qquad 
\substack{
\begin{ytableau}
\, \\ 
\, & \, \\ 
\, & \, & \, & *(pBlue) s \\ 
\, & \, & \, & *(pBlue) \,  & \, \\ 
\, & \, & \, & *(pBlue) \, & \, & \, & \, \\ 
\end{ytableau}
\\
C^s_\lambda
}
\end{equation*}
Given a box $s$, we let $\defin{R^s_\lambda}$ denote the set of boxes in $\lambda$
in the same row as $s$. Similarly, $\defin{C^s_\lambda}$ is the set of boxes in the same column.
Note that this definition makes sense even when $s \notin \lambda$.
\end{definition}

Let $\lambda$ and $\mu$ be Young diagrams. 
Since these are sets of boxes, we may define their intersection, $\defin{\lambda \cap \mu}$ and union $\defin{\lambda \cup \mu}$, 
and these are the Young diagrams for some integer partitions.

\subsection{Jack symmetric functions and structure constants}

There are two popular versions of the Jack symmetric functions; the Jack $P$ functions,
and the \emph{integral form} Jack $J$ functions---the latter has integer coefficients
in the monomial basis. We shall only use the latter functions.
The book \cite{Macdonald:1995} is recommended as a reference on this topic.
For example, the $\jackj_\lambda(\xvec;\alpha)$ for $\lambda \partition 3$,
expanded in the monomial basis are as follows:
\begin{align*}
 \jackj_{3}(\xvec;\alpha) &= (1+3\alpha + 2\alpha^2) \monomial_{3} + (3+3 \alpha) \monomial_{21} + 6\monomial_{111} \\
 \jackj_{21}(\xvec;\alpha) &= (2+ \alpha) \monomial_{21} + 6\monomial_{111} \\
 \jackj_{111}(\xvec;\alpha) &= 6\monomial_{111}.
\end{align*}
The expansions in the power-sum basis are
\begin{align*}
 \jackj_{3}(\xvec;\alpha) &= 2\alpha^2 \powerSum_{3} + 3\alpha \powerSum_{21} + \powerSum_{111} \\
 \jackj_{21}(\xvec;\alpha) &= -\alpha \powerSum_{3} + (\alpha-1) \powerSum_{21} + \powerSum_{111} \\
 \jackj_{111}(\xvec;\alpha) &= 2\powerSum_{3} -3\powerSum_{21} +\powerSum_{111}.
\end{align*}

Let $\defin{\jackjLR_{\mu\nu}^{\gamma}(\alpha)} \in \setQ(\alpha)$ be the Jack Littlewood--Richardson constants, 
determined via the relation
\begin{equation}\label{eq:jackLR}
\jackj_\lambda \jackj_\mu = \sum_\nu \jackjLR_{\lambda\mu}^{\nu}(\alpha) \jackj_\nu.  
\end{equation}
We shall often omit the dependence on $\alpha$, and just write $\jackjLR_{\lambda\mu}^{\nu}$.
Observe that these are closely related to the classical Littlewood--Richardson 
coefficients for Schur functions, denoted $\defin{c_{\lambda\mu}^{\nu}}$.
The Jack polynomials form an orthogonal basis with respect to 
the $\alpha$-deformed Hall inner product, $\langle \,\cdot \, , \,\cdot\, \rangle$.
This is the inner product on symmetric functions with the property that
\[
  \langle \powerSum_\lambda, \powerSum_\mu \rangle =
  \begin{cases}
   z_\lambda \alpha^{\ell(\lambda)} & \text{if $\lambda=\mu$}\\
   0 &\text{otherwise,}
  \end{cases}
\]
and where $n!/\defin{z_\mu}$ is the number of permutations of size $n$ with cycle type $\mu$.

\begin{theorem}[{Jack norm formula, see \cite[5.8]{Stanley:1989}}]
\begin{equation}\label{eq:jacknorm}
\norm{\jackj_\lambda}^2 = \langle \jackj_\lambda, \jackj_\lambda \rangle =\
\prod_{b \in \lambda} \uh{\lambda}{b} \lh{\lambda}{b}.
\end{equation}
\end{theorem}

\subsection{The Stanley Conjectures}\label{sect:stanley}

In \cite{Stanley:1989}, several remarkable conjectures were stated regarding the 
structure of the Jack Littlewood--Richardson coefficients. 
These conjectures are stated in terms of the quantities 
$\langle \jackj_\mu \jackj_\nu , \jackj_\lambda \rangle = \jackjLR_{\mu\nu}^\lambda \norm{\jackj_\lambda}^2$, herein referred to as \emph{Stanley structure coefficients}, 
but can be easily restated in terms of the Littlewood--Richardson coefficients using \eqref{eq:jacknorm}.

\begin{conjecture}[{Stanley conjecture, see \cite[8.3]{Stanley:1989}}] The Stanley structure coefficients are non-negative integer polynomials in $\alpha$, that is,
\begin{equation*}
\langle \jackj_\mu \jackj_\nu , \jackj_\lambda \rangle \in \setZ_{\geq0}[\alpha].
\end{equation*}
\end{conjecture}

In general, we use the terminology of a \emph{strong} form of the Stanley conjecture to refer to any conjecture that proposes an explicit form for $\langle \jackj_\mu \jackj_\nu , \jackj_\lambda \rangle$ which is \emph{manifestly} a non-negative polynomial in $\alpha$. Stanley conjectured such a form in the case $c_{\mu\nu}^\lambda =1$.

\begin{conjecture}[{Strong Stanley Conjecture, see \cite[8.5]{Stanley:1989}}]\label{conj:strongStanley}
If $c_{\mu\nu}^\lambda =1$, then the corresponding Stanley structure coefficient has the form
\begin{equation*}
\langle \jackj_\mu \jackj_\nu , \jackj_\lambda \rangle = 
\left( \prod_{b \in \mu} \tilde{h}_{\mu}(b) \right)
\left( \prod_{b \in \nu} \tilde{h}_{\nu}(b) \right)
\left( \prod_{b \in \lambda} \tilde{h}_{\lambda}(b) \right),
\end{equation*}
where $\tilde{h}_{\sigma}(b)$ is a choice of either $\uh{\sigma}{b}$ or $\lh{\sigma}{b}$ for each box $b$.
Moreover, one chooses $\lh{\,}{b}$ and $\uh{\,}{b}$ exactly $|\lambda|$ times each in the above expression.
\end{conjecture}

Stanley proved Conjecture~\ref{conj:strongStanley} in an important special case.
\begin{theorem}[{Jack Pieri rule, see \cite[6.1]{Stanley:1989}}]
Let $\lambda/\mu$ be a horizontal $r$-strip, then the Stanley 
structure coefficient is given by the following expression
\begin{equation*}
\langle \jackj_\mu  \jackj_{(r)} , \jackj_\lambda \rangle = 
\left( \prod_{b \in \mu} A_{\mu\lambda}(b) \right)
\left( \prod_{b \in (r)}  \uh{(r)}{b}  \right)
\left( \prod_{b \in \lambda} A'_{\mu\lambda}(b) \right),
\end{equation*}
where
\begin{equation*}
A_{\mu\lambda}(b) \coloneqq \begin{cases}
\lh{\mu}{b}  & \text{ if $\lambda/\mu$ does not contain a box in the same column as $b$} \\
\uh{\mu}{b}  &  \text{ otherwise,}
\end{cases}
\end{equation*}
and $A'$ is given by the same formula as $A$, except with upper and lower looks exchanged.
\end{theorem}

In 2013, another case of Conjecture~\ref{conj:strongStanley} was 
proven using a technique involving vertex operators.
\begin{theorem}[Strong Stanley conjecture, rectangular case, see {\cite[4.7]{CaiJing:2013}}]\label{thm:stanleyrectangular}
Let $\mu \subseteq m^n$. 
The Stanley structure coefficient in the rectangular case has the following expression as a product of lower/upper hook lengths:
\begin{equation}\label{eq:stanleyrectangular}
\langle \jackj_\mu  \jackj_{\bmu} , \jackj_{m^n} \rangle =
\left( \prod_{b \in \mu}  \lh{\mu}{b}   \right)
\left( \prod_{b \in \bmu}  \uh{\bmu}{b}  \right)
\left( \prod_{b \in m^n}  D_{\mu,m^n}(b) \right),
\end{equation}
where, for $b\in m^n$, we have defined
\begin{equation}\label{defn:rectD}
\defin{ D_{\mu,m^n}(b)} \coloneqq
\begin{cases}
\uh{m^n}{b} & \text{ if $(b_1,n-1-b_2) \in \mu$} \\
\lh{m^n}{b} &  \text{ otherwise}.
\end{cases}
\end{equation}
\end{theorem}
For example, $\langle J_{211}J_{221},J_{333} \rangle$ is computed using the following hook assignments
\begin{equation*}
\ytableausetup{boxsize=1.1em}
\begin{ytableau}
*(boxL) L   \\
*(boxL) L   \\
*(boxL) L & *(boxL) L
\end{ytableau}\quad
\begin{ytableau}
*(boxU) U  \\
*(boxU) U & *(boxU) U  \\
*(boxU) U & *(boxU) U
\end{ytableau}\quad
\begin{ytableau}
*(boxU) U & *(boxU) U & *(boxL) L \\
*(boxU) U & *(boxL) L & *(boxL) L  \\
*(boxU) U & *(boxL) L & *(boxL) L
\end{ytableau}.
\end{equation*}

In the paper \cite{CaiJing:2013}, the authors provide formula~\eqref{eq:stanleyrectangular}
and note that the answer must be symmetric under the 
exchange of $\mu \leftrightarrow \bmu$, however that symmetry is not explicitly demonstrated in this formula.


In this paper, we provide an explicit demonstration for the symmetry property in formula~\eqref{eq:stanleyrectangular}, by re-proving it using a new technique.
Furthermore, we prove a new case of Conjecture~\ref{conj:strongStanley}, 
the \emph{rectangular union} case, which generalises the 
rectangular case (Thm.~\ref{thm:stanleyrectangular}), and provide an
explicit prescription of the required hook choices. In particular, let $\mu$ be a partition and $m^n$ be 
\emph{any} rectangular partition. Let $\sigma = \mu \cap m^n$, and let $\bsigma$
be the complementary partition of $\sigma$ in $m^n$. 
Note that this reduces to the rectangular case when $\mu \subset m^n$.

\begin{figure}[!htb]
\centering
\includegraphics[scale=1.1,page=5]{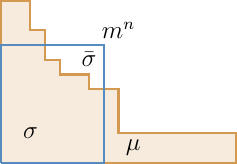}
\caption{
The setup of the rectangular union case.
}
\label{fig:rectunioncaseintro}
\end{figure}

\begin{theorem}[Strong Stanley conjecture, rectangular union case, see {Cor.~\ref{cor:strongStanley}}]\label{thm:rectUnion}
The strong Stanley conjecture holds for $\langle J_\mu J_{\bsigma} , J_{\mu \cup m^n} \rangle$. 
The assignment of lower and upper hooks can be read off from Figure~\ref{fig:rectunionLR}.
\end{theorem}
For example, $\langle J_{42211}J_{211}, J_{43331}\rangle $
is computed using the following hook assignments
\begin{equation*}
\ytableausetup{boxsize=1.1em}
\begin{ytableau}
*(boxL) L   \\
*(boxL) L   \\
*(boxL) L & *(boxL) L \\
*(boxL) L & *(boxL) L \\
*(boxU) U & *(boxU) U & *(boxU) U & *(boxU) U \\
\end{ytableau}\quad
\begin{ytableau}
*(boxU) U  \\
*(boxU) U  \\
*(boxU) U & *(boxU) U
\end{ytableau}\quad
\begin{ytableau}
*(boxU) U \\
*(boxU) U & *(boxU) U & *(boxL) L \\
*(boxU) U & *(boxU) U & *(boxL) L \\
*(boxU) U & *(boxL) L & *(boxL) L \\
*(boxL) L & *(boxL) L & *(boxL) L & *(boxL) L \\
\end{ytableau}. \\
\end{equation*}

The recent paper by Matveev--Wei \cite{Matveev:2023} proves Conjecture~\ref{conj:strongStanley} in the special case when 
we have $c_{\mu\nu}^\lambda = K_{\mu, \lambda-\mu}=1$, where the $K_{\mu, \lambda-\mu}$
is a Kostka coefficient.  
This required equality between Kostka coefficients and Littlewood--Richardson coefficients 
holds in the rectangular case, i.e., the setup in Theorem~\ref{thm:stanleyrectangular} is covered by their work.
We also understand the proof therein to not be constructive, that is, 
no explicit prescription for a choice of upper or lower hooks is presented for each box in the three partitions.

However, our Theorem~\ref{thm:rectUnion} is in general \emph{not} covered by the Matveev--Wei result,
as we sometimes have the strict inequality $c_{\mu\nu}^\lambda < K_{\mu, \lambda-\mu}$ in the rectangular union case. 
In the example above, we have $c^{43331}_{42211,211} =1$
but $K_{42211,43331-211} = K_{42211,22231} = 3$, as evidenced by the three semistandard Young tableax
\[
\ytableausetup{boxsize=1.0em}
 \ytableaushort{1144,22,33,4,5} \quad \ytableaushort{1134,22,34,4,5} \quad  \ytableaushort{1124,23,34,4,5}.
\]

\subsection{A formula involving Jack structure constants}

For any multiset\footnote{A collection with multiplicities.} of boxes $\Gamma$, we define the rational function\footnote{This also depends on $\alpha$.}
$T_{\Gamma}(u) : \setR \to \setR$ by
\begin{equation}\label{def:TGamma}
\defin{ T_{\Gamma}(u) } \coloneqq \prod_{b \in \Gamma} T_\lbox(u-[b])
\text{ where }
\defin{T_\lbox(u)} \coloneqq \frac{(u-[0,0])(u-[1,1])}{(u-[1,0])(u-[0,1])} .
\end{equation}
We give some examples of such functions in Example~\ref{ex:tmu}.

For any two Young diagrams $\lambda, \mu$, we define the \defin{star product} $\lambda \* \mu$ as the multiset of boxes
\[
 \defin{\lambda \* \mu} \coloneqq \bigsqcup_{\substack{s\in \lambda \\ t \in \mu}} s+t.
\]
For example, $31 \* 221$ is computed as
\begin{equation*}
\ytableausetup{boxsize=1.0em} \ydiagram{1,3}\, \*\, \ydiagram{1,2,2} \,=\,\ytableausetup{boxsize=1.0em}
\begin{ytableau}
\, \\ 
2 & 2  & \, \\ 
2 & 3 & 2 & \, \\ 
\, & 2& 2 & \, 
\end{ytableau}
\end{equation*}
where the number inside a box indicates its multiplicity (if $> 1$). 
Lastly, we define
$\defin{\jacktop_\mu} \coloneqq \prod_{b \in \mu^\times} [b]$.

With these tools introduced, we can now state the main motivating result that is used in this work.
\begin{theorem}[\cite{Mickler:2023a}]
Let $\mu,\nu$ be partitions and $n= |\mu|+|\nu|$, 
then the following `sum-product' identity of rational functions (in a complex variable $u$) holds,
involving the Jack Littlewood--Richardson coefficients defined earlier
in \eqref{eq:jackLR}:
\begin{equation}\label{eq:sumproductident}
\sum_{\substack{\gamma\partition  n \\ \mu,\nu \subseteq \gamma } }  \jackjLR_{\mu\nu}^{\gamma}(\alpha) \frac{\jacktop_\gamma}{\jacktop_\mu \jacktop_\nu} \left( \sum_{s \in \gamma \setminus (\mu \cup \nu)} \frac{1}{u-[s]} \right)  = T_{\mu\* \nu}(u) - 1.
\end{equation}
\end{theorem}
For ease of notation, we set
\begin{equation*}
\defin{\hat \jackjLR_{\mu\nu}^{\gamma}} \coloneqq 
\jackjLR_{\mu\nu}^{\gamma}(\alpha) \frac{\jacktop_\gamma}{\jacktop_\mu \jacktop_\nu}.
\end{equation*}

\subsection{Properties of \texorpdfstring{$T_{\mu\*\nu}$}{Tmunu}}
To make use of \eqref{eq:sumproductident}, we will
need to produce some structural
results about the right hand side.
By definition, we have
\[ 
T_{\mu\* \nu}(u)  = \prod_{b \in \mu\* \nu } T_\lbox(u-[b])= \prod_{ \substack{s \in \mu \\ t \in \nu }} T_\lbox(u-[s+t]).
\]
Following from Definition~\ref{def:TGamma}, we may view $T_\lbox(u)$ as the rational
function with zeros at coordinates $(0,0)$ and $(1,1)$, and poles in $(1,0)$
and $(0,1)$. These coordinates are the four vertices of the \emph{box} $(0,0)$.
This allows us to use Young diagrams to visualize the rational function $T_{\mu}(u)$ or $T_{\mu\* \nu}(u)$,
as we see in the following example.

\begin{example}\label{ex:tmu}
With $\mu = 331$, we have that
\[
  T_{\mu}(u) = \frac{(u-[0,0])(u-[1,3])(u-[3,2])}{(u-[0,3])(u-[1,2])(u-[3,0])}.
\]
We can illustrate such a rational function by marking the zeros and poles in the plane.
Let us also set $\nu = 44311$.
Then $T_{\mu}(u)$, $T_{\nu}(u)$ and $T_{\mu \* \nu}(u)$ can be illustrated as
\[
\includegraphics[scale=1.0,page=1]{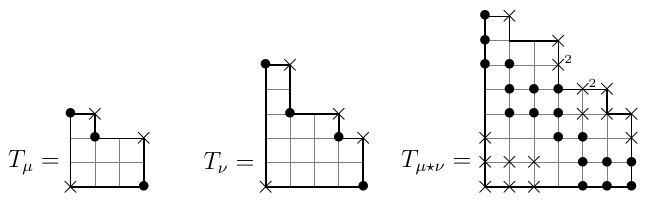}
\]
where $\bullet$ represents a simple pole, 
and $\times$ and $\times^2$ represent a zero of order one and order $2$, respectively.
Note that the number of poles and the number of zeros (counting multiplicity) must be the same in every row (and every column)
of the diagram since any $T_{\Gamma}$
is a product of $T_{\lbox}(u-[b])$ over $b \in \Gamma$.
\end{example}

Note that the poles of $T_{\mu}$ are precisely the elements in $\addset_\mu$,
while the zeros are given by $\remset^+_\mu \cup \{(0,0)\}$.
In other words,
\begin{equation}\label{eq:TmuFormula}
 T_{\mu}(u) = u\cdot \frac{\prod_{s \in \remset_\mu^+ }(u-[s])}{\prod_{t
 \in\addset_\mu}(u-[t ])}.
\end{equation}

\begin{lemma}\label{lem:Torderformula}
For any partitions $\mu$, $\nu$, 
the order of the pole at $u=[s]$ of the function $T_{\mu\* \nu}(u)$
is given by the difference
\begin{equation}\label{eq:Torderformula}
\ord_{u=[s]} T_{\mu\*\nu}(u) = \left| \{  b \in \nu :   s-b \in \addset_\mu \} \right| - \left| \{  b \in \nu : s- b \in ( {\remset^+_\mu} \cup (0,0))\} \right|.
\end{equation}
\end{lemma}
\begin{proof}
We note that $T_{\mu \* \nu}(u) = \prod_{b\in \nu} T_{\mu}(u-[b])$. 
By \eqref{eq:TmuFormula}, the poles of $T_{\mu}(u-[b])$ are at $t+b$ for $t\in \addset_\mu$
and similarly its zeros are at $\bps{t} + b$, for $\bps{t} \in \remset^+_\mu$ and an extra zero at $u=[b]$.
From these observations we can deduce \eqref{eq:Torderformula}.
\begin{figure}[!htb]
\centering
\includegraphics[scale=0.75,page=1]{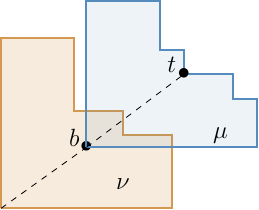}
\hspace{1cm}
\includegraphics[scale=0.75,page=2]{tikzFigures}
\caption{Left: The contributions to poles at $[s]$ of $T_{\mu\*\nu}$.
Here, $b\in \nu$, $t \in \addset_\mu$ and $s = t+b$.
Right:
The contributions to zeroes at $[s]$ of $T_{\mu\*\nu}$.
Here, $b\in \nu$, $t' \in \remset^+_\mu \cup (0,0)$ and $s = \bps{t} + b$.
}\label{fig:contribPoles}
\end{figure}
\end{proof}

\begin{lemma}\label{lemma:Tpoles} Fix $\mu$ and some $s \in \setZ^2$. 
The orders of the poles of $T_{\mu \* \nu}(u)$ are bounded by
\[ 
1- \max( |\addset_\mu|,|\addset_\nu|) \leq \mathrm{ord}_{u=[s]} T_{\mu\*\nu}(u) \leq 1 .
\]
Furthermore, $[\bms{s}]$ is a simple pole of $T_{\mu\*\nu}(u)$  if and only 
if $\bms{s} \notin \nu$ and $\lambda \subseteq \nu$, 
where $\lambda^R$ is the shape between $\mu$ and $s$.
\end{lemma}
\begin{proof}
According to Lemma~\ref{lem:Torderformula}, in order for $T_{\mu\*\nu}(u)$ to have a 
pole $[\bms{s}]$, we must include an outer corner $t\in \addset_\mu$ in $\nu^R \coloneqq \bms{s}-\nu$. 
The smallest shape that $\nu^R$ can be so that it contains $t$ and $\bms{s}$ is the rectangle 
of boxes between $t$ and $s$, denoted $\Rect(t,s)$. 
However, $\Rect(t,s)$ also includes the two (shifted) inner corners $m_1, m_2 \in \remset^+_\mu$ adjacent to $t$, 
as seen in Figure~\ref{fig:rectcontainsmins}.
\begin{figure}[!htb]
\centering
\includegraphics[page=1,scale=0.75]{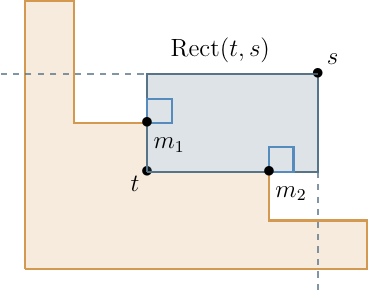}
\caption{
Whenever we include an outer corner $t\in \addset_\mu$ in $\bms{s}-\nu$, we must also 
include the adjacent inner corners $m_1, m_2 \in \remset^+_\mu$.
}
\label{fig:rectcontainsmins}
\end{figure}
Thus according to Lemma~\ref{lem:Torderformula}, 
$T_{\mu\* \Rect(0,s-t)}(u)$ will have a \emph{zero} of order $1$ at $u=[\bms{s}]$. 
We repeat this inclusion of rectangles for each inner corner of $\mu$ that is inside $\Rect(s)$. 
This process yields $\lambda^R \coloneqq \cup_{t \in (\addset_\mu \cap \Rect(s))} \Rect(t,s)$, which is 
precisely the shape between $\mu$ and $s$. It is clear from Figure~\ref{fig:polestructure} that $\lambda^R$
contains one more box from $\addset_\mu$ than from $\remset^+_\mu$, and thus $T_{\mu\* \lambda}(u)$ 
has a pole at $[\bms{s}]$. Thus, if $\nu^R \supset \lambda^R$, we should have a pole of degree $1$. 
However, this pole is annihilated if $\nu^R$ also includes the zero coming from $(0,0) \in \mu$, that is, if $\bms{s} \in \nu$.

\end{proof}
\begin{figure}[!htb]
\centering
\includegraphics[scale=0.75,page=2]{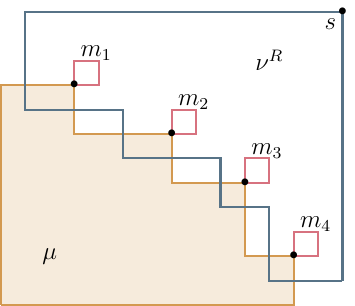}
\qquad %
\includegraphics[scale=0.75,page=3]{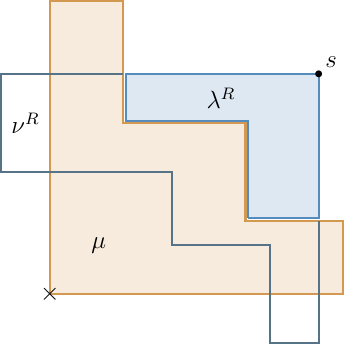}
\caption{
On the left, the largest degree zero is obtained when $\nu^R \coloneqq \bms{s}-\nu$ contains 
all of $\remset^+_\mu$, and avoids $\addset_\mu$. 
On the right, a simple pole occurs when $\nu^R$ contains $\lambda^R$ and avoids $(0,0)$.
}
\label{fig:polestructure}
\end{figure}

To describe the corners to the left (or right) of a particular
corner $s=(s_1,s_2)$, we introduce the notation
\[
\defin{\remset_\sigma^{<s}} \coloneqq  \{ t = (t_1,t_2) \in \remset_\sigma : t_1 < s_1 \},
\]
and mutatis mutandis for $\defin{\remset_\sigma^{>s}}$,
$\defin{\addset_\sigma^{<s}}$ and $\defin{\addset_\sigma^{>s}}$.
One can check directly the following reinterpretations
of differences of corners as hook vectors.
\begin{lemma}\label{lem:cornerdiffashooks}
For fixed $s \in \remset_\sigma$, the following hold:

For $t \in \remset_\sigma^{<s}$,
we have $\bps{s} - \bps{t} = \uhv{\sigma \setminus s}{\bmin{s}{t}}$.

For $t \in \addset_\sigma^{<s}$, we have $\bps{s} - t = \uhv{\sigma}{\bmin{s}{t}}$.

For $t \in \remset_\sigma^{>s}$, we have $\bps{s} - \bps{t} = -\lhv{\sigma \setminus s}{\bmin{s}{t}}$.

For $t \in \addset_\sigma^{>s}$, we have $\bps{s} - t = -\lhv{\sigma}{\bmin{s}{t}}$.
\end{lemma}
The proof of Lemma~\ref{lem:cornerdiffashooks} is straighforward and
we simply refer to Figure~\ref{fig:cornerdiffashooks}.
\begin{figure}[!htb]
\centering
\includegraphics[scale=1.0,page=3]{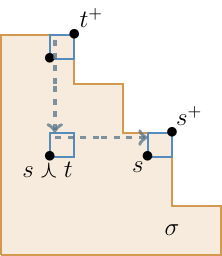}
\qquad
\includegraphics[scale=1.0,page=4]{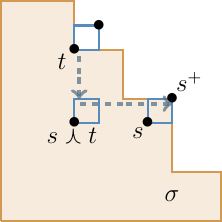}
\caption{
The difference between partition corners are hook lengths.
In the first figure, $\bps{s} - \bps{t} = \uhv{\sigma\setminus s}{\bmin{s}{t}}$, while in the second figure we have $\bps{s} - t = \uhv{\sigma}{\bmin{s}{t}}$. The other two cases of Lemma~\ref{lem:cornerdiffashooks} are verified in a similar manner.
}
\label{fig:cornerdiffashooks}
\end{figure}

\begin{lemma}[Expansion Lemma]\label{lemma:expansion}
Let $\sigma$ be a partition, and pick $s \in \remset_\sigma$.
Then $T_{\sigma}(u)$ can be expanded\footnote{Alternatively, one may shift the argument to write the simpler expression,
\begin{equation*}
T_{\sigma}(u+[\bps{s}]) = u (u+[\bps{s}]) 
\cdot
\tfrac{
\prod_{b \in R^s_{\sigma \setminus s}} \left( u+  \uh{\sigma  \setminus s}{b}  \right) }{
\prod_{b \in R^s_\sigma } \left( u+ \uh{\sigma}{b} \right)
}
\cdot
\tfrac{
\prod_{b \in C^s_{\sigma \setminus s}  }\left(u- \lh{\sigma \setminus s}{b} \right) }{
\prod_{b \in C^s_{\sigma}}\left(u-\lh{\sigma}{b} \right)}.
\end{equation*}
} as a row-column product \defin{relative to $s$} as
\begin{equation*}
T_{\sigma}(u) = u \cdot \left(u-[\bps{s}]\right)
\cdot
\tfrac{
\displaystyle\prod_{b \in R^s_{\sigma \setminus s}} \left( u-[\bps{s} - \uhv{\sigma  \setminus s}{b} ] \right) }{
\displaystyle\prod_{b \in R^s_\sigma } \left( u-[\bps{s} - \uhv{\sigma}{b} ] \right)
}
\cdot
\tfrac{
\displaystyle\prod_{b \in C^s_{\sigma \setminus s}  }\left(u-[\bps{s} + \lhv{\sigma \setminus s}{b}] \right) }{
\displaystyle\prod_{b \in C^s_{\sigma}}\left(u-[\bps{s} + \lhv{\sigma}{b}] \right)}.
\end{equation*}

Each of the four factors in the product precisely captures the zeros and poles in each of the four quadrants with a corner in $s$,
see Figure~\ref{figure:expansionlemma}.
\begin{figure}[!htb]
\centering
\includegraphics[scale=1.2,page=2]{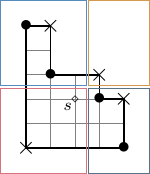}
\caption{The four quadrants of $T_{\mu}(u)$ expanded around $u=[s]$ for $s\in \remset_\mu$.}
\label{figure:expansionlemma}
\end{figure}

\end{lemma}
\begin{proof}
Recall from \eqref{eq:TmuFormula} that
\begin{align*}
T_{\sigma}(u) &= u\frac{\prod_{t \in \remset_\sigma }(u-[\bps{t}])}{\prod_{t
 \in\addset_\sigma }(u-[t ])}  \\
 &= u\frac{\prod_{t \in  \remset_\sigma } (u-[\bps{s} - (\bps{s}-\bps{t})])}{\prod_{t
 \in\addset_\sigma }(u-[\bps{s} - (\bps{s}-t)])}.
\end{align*}
Using Lemma~\ref{lem:cornerdiffashooks}, we split this according to $\remset_\sigma = \remset_\sigma^{<s} \cup \{ s\}\cup \remset_\sigma^{>s}$, as
 \[
 =
 u(u-[\bps{s}])
 \frac{
 \prod_{t \in \remset_\sigma^{<s}}
 \left(u-[\bps{s} - \uhv{\sigma \setminus s}{ \bmin{s}{t}  }]\right)
 }{
 \prod_{t \in \addset_\sigma^{<s}}
 \left(u-[\bps{s}- \uhv{\sigma}{  \bmin{s}{t} }]\right)
 }
 \frac{
 \prod_{t \in \remset_\sigma^{>s}}
 \left(u-[\bps{s} + \lhv{\sigma \setminus s}{ \bmin{s}{t} }]\right)
 }{
 \prod_{t \in \addset_\sigma^{>s}}
 \left(u-[\bps{s}+ \lhv{\sigma}{ \bmin{s}{t} } ]\right)
 }.
\]
Note that in the first fraction, $\bmin{s}{t}$
is always a box in the same row as $s$,
while in the second fraction, $\bmin{s}{t}$
is a box in the same column as $s$.
Next, we claim the that following holds for the first ratio of products:
\begin{equation}\label{eq:cornersumtorowsum}
\frac{
\prod_{t \in \remset_\sigma^{<s}}
(u-\uh{\sigma \setminus t}{\bmin{s}{t} })}
{
\prod_{t \in \addset_\sigma^{<s}}
(u-\uh{\sigma}{\bmin{s}{t} })} = \frac{\prod_{b \in R^s_{\sigma \setminus s}  }(u-\uh{\sigma \setminus s}{b})}{
\prod_{b \in R^s_\sigma  }(u-\uh{\sigma}{b})} .
\end{equation}
That is, the ratio of products which only receives contributions from
corners can be extended to include all boxes in the row of $s$.
This holds because for two adjacent boxes $b, b+(1,0)$ in $R_\sigma^s$, with $b$ not in the same column as an inner corner, we have
\[
\uhv{\sigma\setminus s}{b} =  \uhv{\sigma}{b+(1,0)}.
\]
Thus there is pairwise cancellation between pairs of boxes in the right
hand side of \eqref{eq:cornersumtorowsum},
unless $b$ shares a column with an inner corner, in which case $b+(1,0)$ shares a column with the adjacent outer corner, see Figure~\ref{figure:adjacentboxes}.
\begin{figure}[!htb]
\centering
\includegraphics[page=1,scale=0.75]{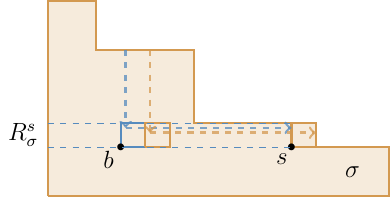}
\hspace{5mm}
\includegraphics[page=2,scale=0.75]{hookSame}
\caption{Relation between $\uhv{\sigma\setminus s}{b}$ and $\uhv{\sigma}{b+(1,0)}$.
In the first diagram, we have that $\uhv{\sigma\setminus s}{b} = \uhv{\sigma}{b + (1,0)}$,
but this does not hold in the second diagram as $b$ shares a column with an inner corner.
}
\label{figure:adjacentboxes}
\end{figure}

We find a similar result for the left hand ratio of products,
this time with products over columns:
\begin{equation*}\label{eq:cornersumtocolsum}
\frac{
\prod_{t \in \remset_\sigma^{>s}}
(u-\lh{\sigma \setminus t}{\bmin{s}{t} })}
{
\prod_{t \in \addset_\sigma^{>s}}
(u-\lh{\sigma}{\bmin{s}{t} })} = \frac{\prod_{b \in C^s_{\sigma \setminus s}  }(u-\lh{\sigma \setminus s}{b})}{
\prod_{b \in C^s_\sigma  }(u-\lh{\sigma}{b})}.
\end{equation*}
Piecing these two results together this becomes
 \[
 = u(u-[\bps{s}]) \frac{\prod_{b \in R^s_{\sigma \setminus s}  }(u-[\bps{s} - \uhv{\sigma \setminus s}{b}])}{\prod_{b \in R^s_\sigma  }(u-[\bps{s} - \uhv{\sigma}{b}])}\frac{\prod_{b \in C^s_{\sigma \setminus s}  }(u-[\bps{s} + \lhv{\sigma \setminus s }{b}])}{\prod_{b \in C^s_\sigma  }(u-[\bps{s}+ \lhv{\sigma}{b}])}.
 \]
\end{proof}

\section{Rectangular Case}

Let $c_{\mu\nu}^{\lambda}$  denote the Schur Littlewood--Richardson coefficients. One can easily show the following result.

\begin{lemma}
If $\lambda=m^n$ is a rectangular partition, then
the (Schur) Littlewood--Richardson coefficient
satisfies $c_{\mu\nu}^{\lambda} \in \{0,1\}$,
and takes the value $1$ if and only $\mu \subset m^n$ and $\nu = \bmu$.
\end{lemma}

Among all integer partitions $\lambda \vdash mn$,
the rectangular partition $m^n$ is the unique partition
containing the box with coordinate $\bms{v}=(m-1,n-1)$.
This simple fact allows us to extract the
rectangular Jack Littlewood--Richardson coefficients
from the formula \eqref{eq:sumproductident}.

\begin{corollary}\label{corr:rectangularcase}
Let  $\mu \subseteq m^n$. Then
\begin{equation*}
\langle \jackj_{\mu}\jackj_{\bmu},\jackj_{m^n} \rangle =
\left( \Res_{u=[\bms{v}]}  T_{\mu\*\bmu}(u) \right) \cdot
\norm{ \jackj_{m^n} }^2 \cdot \frac{ \jacktop_\mu \jacktop_{\bmu}} {\jacktop_{m^n}}.
\end{equation*}
\end{corollary}

To evaluate the residue in Corollary~\ref{corr:rectangularcase},
we will need to build up a sequence of structural results about $T_{\mu\*\bmu}(u)$.
We work always in reference to the fixed rectangle $m^n$. 
For any box $b$, we define the complementary box $\defin{\bb} \coloneqq \bms{v} - b$. For every $\sigma \subset m^n$, the map $t \to \bar t$, maps $\remset_\sigma$ into $\addset_{\bsigma}$, and the image misses only a single outer corner of $\bsigma$ at the edge.

\begin{lemma}\label{mirrorrules}
Let $\sigma \subset m^n$ and recall that $\bsigma$
represents the complementary partition of $\sigma$
inside the rectangle.
Fix the inner corner $t \in \remset_\sigma$ and let
$\bt \in \addset_{\bsigma}$ be the
corresponding outer corner of $\bar{\sigma}$.
\begin{itemize}
 \item
For $s \in \addset_{\sigma}$,
we have $\eh{\sigma}{\bmin{s}{t}} = \eh{\bsigma}{ \bmin{\bs}{\bt} }$;
\item
for $w \in \remset_\sigma$,
we have $\eh{\sigma\setminus t}{ \bmin{w}{t} } = \eh{\bsigma\cup\bt}{ \bmin{\bw}{\bt} }$,
\end{itemize}
where $\defin{\eh{\sigma}{b}}$ is either $\uh{\sigma}{b}$ or $\lh{\sigma}{b}$.
See Figure~\ref{fig:hookcomplement}.
\end{lemma}

\begin{figure}[!htb]
\centering
\includegraphics[page=4,scale=0.9]{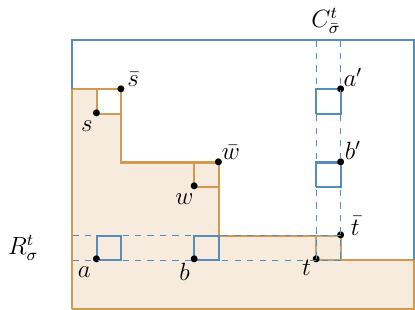}
\hfill
\includegraphics[page=5,scale=0.9]{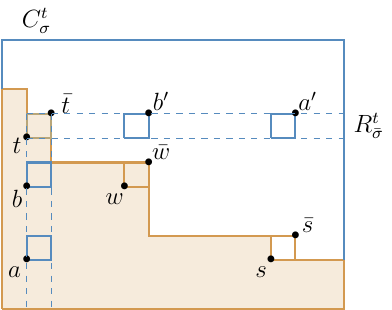}
\caption{
The hook complement for boxes to the left of $t$ and to the right of $t$, respectively.
In the left figure, we see for example that
$\uh{\sigma\setminus t}{\bmin{w}{t}} = \uh{\bsigma \cup \bt}{\bmin{\bw}{\bt}}$ and in the right figure we have
for example $\uh{\sigma}{\bmin{s}{t}} = \uh{\bsigma}{\bmin{\bs}{\bt}}$.
}
\label{fig:hookcomplement}
\end{figure}
The next lemma conveys how we can `flip' row/column box
products in $\sigma$ to column/row box products in $\bsigma$.
\begin{lemma}[Flip rules]\label{lemma:flip}
For $t \in \remset_\sigma$, we have
\begin{align}\label{eq:fliprules}
 \frac{\prod_{b \in C^t_{\sigma \setminus t}  }(u-\uh{\sigma \setminus t}{b})}{
\prod_{b \in C^t_\sigma  }(u-\uh{\sigma}{b})} &= \frac{1}{(u-[(m,0)-t])}
\frac
{\prod_{b \in R^{\bt}_{\bsigma}  }(u- \uh{\bsigma \cup \bt}{b})}
{\prod_{b \in R^{\bt}_{\bsigma} }(u-\uh{\bsigma}{b} )}, \\
\frac{
\prod_{b \in R^t_{\sigma \setminus  t}  }(u- \lh{\sigma \setminus  t}{b}   )}{
\prod_{b \in R^t_\sigma  }(u-\lh{\sigma}{b} )} &= \frac{1}{(u-[t-(0,n)]) }
\frac
{\prod_{b \in C^{\bt}_{\bsigma} }(u- \lh{\bsigma \cup \bt}{b} )}
{\prod_{b \in C^{\bt}_{\bsigma} }(u- \lh{\bsigma}{b} )}.
\end{align}
\end{lemma}
\begin{proof}
From the proof of the expansion Lemma~\ref{lemma:expansion}, we know that the ratio of column products can be expressed as
\begin{equation}\label{eq:preflipped}
 \frac{\prod_{b \in C^t_{\sigma \setminus t}  }(u-\uh{\sigma \setminus t}{b})}{
\prod_{b \in C^t_\sigma  }(u-\uh{\sigma}{b})} = \frac{\prod_{s \in \remset_{\sigma}^{>t}}(u-\uh{\sigma \setminus t}{\bmin{s}{t} })}{
\prod_{s \in \addset_{\sigma}^{>t} }(u-\uh{\sigma}{\bmin{s}{t} })} .
\end{equation}

We now use Lemma~\ref{mirrorrules} to flip each of the terms in this product. All the inner corners to the right of $t$ can be flipped this way, however if $(m,0) \in \addset_\sigma^{>t}$, that is, if $\sigma$ runs to the right edge of $m^n$, then there is no corresponding element in $\remset_{\bsigma}$. We treat this case first, and then consider case when $(m,0) \notin \addset_\sigma^{>t}$. We claim that in both cases the following holds:

\begin{equation}\label{eq:flippoles}
 \frac{\prod_{s \in \remset_{\sigma}^{>t}}(u-\uh{\sigma \setminus t}{\bmin{s}{t}})}{
\prod_{s \in \addset_{\sigma}^{>t} }(u-\uh{\sigma}{\bmin{s}{t}})} = \frac{1}{(u-[(m,0)-t])}\frac
{\prod_{\bs \in \addset_{\bsigma}^{< \bt} }(u- \uh{\bsigma \cup  \bt}{\bmin{\bs}{\bt}})}
{\prod_{ \bs \in \remset_{\bsigma}^{<\bt} }(u-\uh{\bsigma}{\bmin{\bs}{\bt}} )}.
\end{equation}

In the case $(m,0) \in \addset_\sigma^{>t}$ we have contribution in the LHS denominator coming from $\bmin{t}{(m,0)}= (t_1,0) $.
As this term cannot be flipped, it contributes an extra pole on the RHS at
\[
\uhv{\sigma}{(t_1,0)} = (m-t_1,-t_2) = (m,0)-t,
\]
which appears in \eqref{eq:flippoles}.

In the case where $(m,0) \notin \addset_\sigma^{>t}$, we must now have $(0,n) \in \addset_{\bsigma}^{<\bt}$ and that this corner of $\bsigma$ has no corresponding corner in $\sigma$. In the RHS numerator of \eqref{eq:flippoles} we find a contribution from $\bmin{\bt}{(0,n)} = (0,\bt_2)$, for which
\[\uhv{\bsigma+\bt}{(0,\bt_2)} = (\bt_1 +1, -(n-1-\bt_2)) =  \bps{\bt}-(0,n) = (m,0)-t.\]
As this zero in the numerator of the RHS of \eqref{eq:flippoles} does not appear as a flipped hook from the left hand side, this extra zero will have to be cancelled by the same pole factor that appeared in first case. Thus \eqref{eq:flippoles} holds in both cases, and we can then extend the products in the right hand side out over the full row $R_{\bsigma}^{\bt}$ (using a formula similar to equation \eqref{eq:preflipped}) to recover equation \eqref{eq:fliprules}. By a similar technique, we find the other flip rule for row products.
\end{proof}

The extra pole factors entering into the formulas of the
flip Lemma~\ref{lemma:flip} have an alternate interpretation
in terms of hooks in the rectangular partition $m^n$.

\begin{definition}\label{def:vhflips}
Recall that we have a fixed rectangular partition $m^n$ in mind.
For $b=(b_1,b_2)$ we denote $\defin{\hflip{b}} \coloneqq (m-1-b_1,b_2)$
and $\defin{\vflip{b}} \coloneqq (b_1,n-1-b_2)$.
\end{definition}

\begin{lemma}\label{lemma:recthook}
We have that
\[
(m,0) - b = \uhv{m^n}{\vflip{b}}
\quad \text{and} \quad
b - (0,n) = \lhv{m^n}{\hflip{b}}.
\]
\end{lemma}
\begin{proof}See Figure~\ref{fig:boxhooks}.\end{proof}

\begin{figure}[!htb]
\centering
\includegraphics[scale=1,page=1]{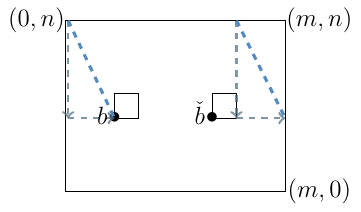}
\hspace{0.1cm}
\includegraphics[scale=1,page=2]{bhat}
\caption{One can easily see that $(m,0) - b = \uhv{m^n}{\vflip{b}}$
and $b - (0,n) = \lhv{m^n}{\hflip{b}}$.
}\label{fig:boxhooks}
\end{figure}

We now arrive at the main structural result, which expresses the rational
function $T_{\mu\*\bmu}(u)$ as a product involving hook lengths.

\begin{theorem}
The function $T_{\mu\*\bmu}(u)$ is given by the following expression\footnote{As before, one may shift the argument to write a simplified expression,
\begin{equation*}\label{mumudecompalt}
T_{\mu\*\bmu}(u+[\bms{v}]) = \prod_{b \in \mu} \frac{u+[\bb]}{u+[b] }
\prod_{b \in \bmu} \frac{u+ \lh{\bmu}{b} }{u+\lh{m^n}{{\vflip{ b}}}}
\prod_{b \in \mu} \frac{u-\uh{\mu}{ b}}{u-\uh{m^n}{\hflip{b}}}.
\end{equation*}
},
\begin{equation}\label{eq:mumudecomp}
T_{\mu\*\bmu}(u) = \prod_{b \in \mu} \frac{u-[b]}{u-[\bb] }
\prod_{b \in \bmu} \frac{u-[\bms{v} - \lhv{\bmu}{b} ]}{u-[\bms{v}-\lhv{m^n}{{\vflip{ b}}}]}
\prod_{b \in \mu} \frac{u-[\bms{v}+\uhv{\mu}{ b}]}{u-[\bms{v}+\uhv{m^n}{\hflip{b}}]},
\end{equation}
Furthermore, each of the three products in the expression \eqref{eq:mumudecomp} is the component of the LHS explicitly decomposed with respect to quadrants around $\bms{v}$, as in Fig.~\ref{threequadrants}. Thus, each of these factors, as a function of a partition $\mu \subset m^n$, is manifestly invariant under $\mu \to \bmu$.
\end{theorem}

\begin{figure}[!htb]
\centering
\includegraphics[scale=1.25,page=3]{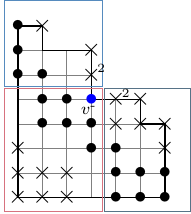}
\caption{The three quadrants of $T_{\mu\*\bmu}(u)$ around $u=[\bms{v}]$. Note that each region contains an equal number of poles and zeros}
\label{threequadrants}
\end{figure}

\begin{proof}
We prove this by induction on the number of boxes in $\mu$.
We start with the base case $\mu_0 = \{1\}$ and so $\bar\mu_0=m^n\setminus\bms{v}$, where we find
\begin{eqnarray*}
T_{\mu_0\*\bmu_0}(u) &=& T_{1\*(m^n \setminus \bms{v} )}(u)\\
 &=& \frac{u-[0,0]}{u-[m-1,n-1]} \frac{u-[m-1,n]}{u-[0,n]} \frac{u-[m,n-1]}{u-[m,0]}.
\end{eqnarray*}
Of the three terms in this product, the first and last agree with the statement of the theorem, with the product over the the only box $b=(0,0)$ of $\mu_0 = \{1\}$. For the middle factor, we need to check
\begin{equation*}
\prod_{b \in \bmu_0} \frac{u-[\bms{v} - \lhv{\bmu_0}{b} ]}{u-[\bms{v}-\lhv{m^n}{{\vflip{ b}}}]} =  \frac{u-[m-1,n]}{u-[0,n]}.
\end{equation*}
To show this holds, we first fill out the bottom product from $\bmu_0 \to m^n$
\[ \prod_{b \in \bmu_0} \tfrac{u-[\bms{v} - \lhv{\bmu_0}{b} ]}{u-[\bms{v}-\lhv{m^n}{{\vflip{ b}}}]} =(u-[\bms{v}-\lhv{m^n}{{\widehat{\bms{v}}}}])\tfrac{\prod_{b \in \bmu_0} u-[\bms{v} - \lhv{\bmu_0}{b} ]}{\prod_{b \in m^n} u-[\bms{v}-\lhv{m^n}{{b}}]}.  \] 
Next, we similarly fill out the numerator product by carefully expanding the row/columns containing $\bms{v}$, which yields
\[ \prod_{b \in \bmu_0} u-[\bms{v} - \lhv{\bmu_0}{b} ] = \frac{(u-[\bms{v} - \lhv{m^n}{\bms{v}}]) \prod_{b \in m^n} (u-[\bms{v} - \lhv{m^n}{b} ])}{(u-[\bms{v} - \lhv{m^n}{(0,n-1)} ])(u-[\bms{v} - \lhv{m^n}{(m-1,0)} ])} . \]
Combining these previous two results, we find
\[ \prod_{b \in \bmu_0} \frac{u-[\bms{v} - \lhv{\bmu}{b} ]}{u-[\bms{v}-\lhv{m^n}{{\vflip{ b}}}]}  = \frac{(u-[\bms{v} - \lhv{m^n}{\bms{v}} ])(u-[\bms{v}-\lhv{m^n}{{\widehat{\bms{v}}}}])}{(u-[\bms{v}- \lhv{m^n}{(0,n-1)} ])(u-[\bms{v} - \lhv{m^n}{(m-1,0)} ])}.  \]
Now using $\widehat{\bms{v}} = (m-1,0)$, $\lhv{m^n}{\bms{v}} = (0,-1)$, and $\lhv{m^n}{(0,n-1)} = (m-1,-1)$, we arrive at
\[\prod_{b \in \bmu_0} \frac{u-[\bms{v} - \lhv{\bmu}{b} ]}{u-[\bms{v}-\lhv{m^n}{{\vflip{ b}}}]} =  \frac{u-[\bms{v} -(0,-1)  ]}{u-[\bms{v} - (m-1,-1) ]} =  \frac{u-[m-1,n]}{u-[0,n]}. \] 
Thus we have shown the base case.
Next, for the inductive step, we assume the theorem holds for $\mu$, and we show it holds for $\mu+s$. Here, we have $\overline{\mu+s} = \bmu - \bs$ where  $s+\bs = \bms{v}$.
We then find
\begin{eqnarray*}
T_{\mu+s\*\overline{\mu+s}}(u) &=& T_{\mu+s\*(\bmu - {\bs})}(u)\\
&=&  \frac{ T_{s\* \bmu}(u)}{ T_{(\mu+s)\*{\bs}}(u)}T_{\mu\*\overline{\mu}}(u)\\
&=& \frac{ T_{\bmu}(u-[s])}{ T_{\mu+{s}}(u-[\bs])}  T_{\mu\*\overline{\mu}}(u).
\end{eqnarray*}
We now simplify this factor in the last line by using the
expansion in Lemma~\ref{lemma:expansion} to expand
the numerator at $\bs \in \remset_{\bmu}$
and the denominator at $s \in \remset_{\mu+s}$:
\[
T_{\bmu}(u-[s]) = (u-[s])(u-[v])
\tfrac{\prod_{b \in R^{\bs}_{\bmu-\bs}  }(u-[v- \uhv{\bmu-\bs}{b} ])}{\prod_{b \in R^{\bs}_{\bmu}  }(u-[v-\uhv{\bmu}{b}])}
\textcolor{blue}{ \tfrac{\prod_{b \in C^{\bs}_{\bmu-\bs}  }(u-[v+\lhv{\bmu-\bs}{b}])
}{\prod_{b \in C^{\bs}_{\bmu} }(u-[v+\lhv{\bmu}{b}])}}.    \]
\[ T_{ \mu+s}(u-[\bs]) = (u-[\bs])(u-[v])
\textcolor{blue}{ \tfrac{\prod_{b \in R^s_{\mu} }(u-[v-\uhv{\mu}{b}])}{
\prod_{b \in R^s_{\mu+s}  }(u-[v-\uhv{\mu+s}{b}])}}  
\tfrac{\prod_{b \in C^s_\mu  }(u-[v+ \lhv{\mu}{b}])}{\prod_{b \in C^s_{\mu+s}  }(u-[v+\lhv{\mu+s}{b}])}.  
\] 

We use Lemma~\ref{lemma:flip} to flip the blue factors
to the following red factors,
and use Equation~\ref{eq:lrhookshift},
to arrive at the expressions
\[
T_{\bmu}(u-[s]) = \tfrac{(u-[s])(u-[v])}{
\textcolor{red}{(u-[v+(m,0)-\bps{\bs}])}}
\tfrac{\prod_{b \in R^{\bs}_{\bmu-\bs}  }(u-[v-\uhv{\bmu-\bs}{b}])}{
\prod_{b \in R^{\bs}_{\bmu}  }(u-[v-\uhv{\bmu}{b}])}
\textcolor{red}{ \tfrac{\prod_{b \in R^s_{\mu} }(u-[v+\lhv{\mu+s}{b}])}{
\prod_{b \in R^s_{\mu} }(u-[v+\lhv{\mu}{b}])} },    
\] 
\[
T_{ \mu+s}(u-[\bs]) = \tfrac{(u-[\bs])(u-[v])}{\textcolor{red}{(u-[v-(\bps{s}-(0,n))])}}
\textcolor{red}{  \tfrac{\prod_{b \in C^{\bs}_{\bmu-\bs}  }(u-[v-\uhv{\bmu}{b}])}{
\prod_{b \in C^{\bs}_{\bmu-\bs}  }(u-[v-\uhv{\bmu-\bs}{b}])} }
\tfrac{\prod_{b \in C^s_\mu  }(u-[v+\lhv{\mu}{b}])}{\prod_{b \in C^s_{\mu+s}  }(u-[v+\lhv{\mu+s}{b}])}.  
\] 
The ratio of these terms then fills out to products over the entire partitions $\mu$, $\bmu$, as the $s$-row-column products are the only ones that are changed by adding in $s$, and so we arrive at the following expression,
\begin{eqnarray}\label{eq:bigprod}
\tfrac{T_{\bmu}(u-[s])}{T_{ \mu+s}(u-[\bs])} &=& 
\tfrac{ (u-[s]) }{ (u-[\bs]) } \tfrac{ (u-[\bms{v}- (s-(0,n))]) }{ (u-[\bms{v}+(m,0)-\bs]) } \\
&&
\times \tfrac{\prod_{b \in \bmu-\bs  }(u-[v-\uhv{\bmu-\bs}{b}])}{\prod_{b \in {\bmu} }(u-[v-\uhv{\bmu}{b}])}\tfrac{\prod_{b \in \mu+s  }(u-[v+\lhv{\mu+s}{b}])}{\prod_{b \in \mu  }(u-[v+\lhv{\mu}{b}])}.
\end{eqnarray}
We consider each of these four terms separately, and note their effect on the formula~\eqref{eq:mumudecomp}.
For the first term of formula~\ref{eq:bigprod}, it correctly updates the contribution from $s$ in the first term in formula~\eqref{eq:mumudecomp}, i.e.
\begin{equation*}
\frac{ (u-[s])}{(u-[\bs])} \prod_{b \in \mu} \frac{u-[b]}{u-[\bb] } = \prod_{b \in \mu \cup s} \frac{u-[b]}{u-[\bb] }.
\end{equation*}
The second term of formula~\eqref{eq:bigprod} may be rewritten using Lemma~\ref{lemma:recthook}, as
\begin{equation*}\frac{ u-[\bms{v}-\lhv{m^n}{\hflip{s}}] }{ u-[\bms{v}+\uhv{m^n}{\vflip{\bs}}] }. \end{equation*}
If we note that $\vflip{\bs} = \hflip{s}$, we see that this term is precisely what is needed to include the $s$ contribution (resp. remove the $\bs$ contribution) in the denominator of the third (resp. second) term in formula~\eqref{eq:mumudecomp}.
The last two terms in formula~\eqref{eq:bigprod}  update the numerators of the last two terms in formula~\eqref{eq:mumudecomp} respectively.
Thus we see that the factor of $\tfrac{T_{\bmu}(u-[s])}{T_{ \mu+s}(u-[\bs])}$ precisely is what is required for the inductive step to work.

Note: As the function $T_{\mu\*\bmu}$ is manifestly symmetric under $\mu \leftrightarrow \bmu$, it clear that formula~\eqref{eq:mumudecomp} is also symmetric under this exchange. As a consequence, each of the three factors in formula~\eqref{eq:mumudecomp} must also be symmetric, as they are expressed explicitly as products of poles/zeros (of the form $u-[\bms{v} \pm \lhv{}{b}]$) of the function $T_{\mu\*\bmu}(u)$ that lie in a distinct quadrant of the $u$-plane relative to $[\bms{v}]$.
\end{proof}

With this we can provide a new proof of a result first shown by \cite[4.7]{CaiJing:2013}.
\begin{corollary}
The strong Stanley conjecture holds in the rectangular case.
In particular, the rectangular Littlewood--Richardson coefficient is given by
\begin{equation}\label{eqn:rectangularLR}
\jackjLR_{\mu,\bmu}^{m^n}  =  \frac{ \prod_{b \in \mu}\lh{\mu}{b} \prod_{b \in \bmu}  \uh{\bmu}{b} }{\prod_{b \in m^n} D'_{\mu}(b) },
\end{equation}
where\footnote{Note here that the function $D'$ differs from $D$ given in formula~\eqref{defn:rectD} by flipping the hooks $U \leftrightarrow L$, since we are writing an expression for the LR coefficients rather than the Stanley inner products.}
\begin{equation*}
D'_{\mu,m^n}(b) \coloneqq
\begin{cases}
\lh{m^n}{b} & \text{ if $\vflip{b} \in \mu$} \\
\uh{m^n}{b} &  \text{ otherwise}.
\end{cases}
\end{equation*}
\end{corollary}

\begin{figure}[!htb]
\centering
\includegraphics[width=0.6\textwidth,page=4]{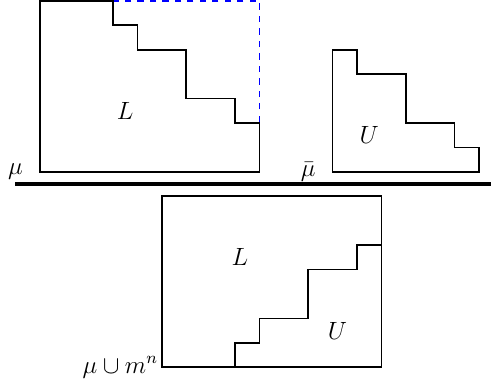}
\caption{The rectangular case of the Littlewood--Richardson coefficient, $\jackjLR_{\mu, \bmu}^{m^n}$}
\label{fig:rectcase}
\end{figure}

\begin{proof}
As stated in Corollary~\ref{corr:rectangularcase}, we have
\begin{equation*}
\jackjLR_{\mu,\bmu}^{m^n} =\frac{\jacktop_\mu \jacktop_{\bmu}} {\jacktop_{m^n}}  \Res_{u=[\bms{v}]} T_{\mu\*\bmu}(u).
\end{equation*}
Using formula~\ref{eq:mumudecomp}, we find
\begin{equation}\label{eq:resexpression}
\Res_{u=[\bms{v}]}  T_{\mu\*\bmu}(u) =
\frac{\prod_{b \in \mu} [\bb]}{\prod_{b \in \mu^\times} [b] } 
\prod_{b \in \bmu} \frac{\lh{\bmu}{b} }{\lh{m^n}{\vflip{b}}} \prod_{b \in \mu} \frac{\uh{\mu}{b}}{\uh{m^n}{\hflip{b}}}.
\end{equation}
As noted earlier, we know that each of the three factors in \eqref{eq:resexpression} remains
invariant under the switch $\mu \leftrightarrow \bmu$, so we are free to
switch each of the last two factors. 
This provides a direct proof that this result is symmetric in $\mu$ and $\bmu$, and gives an alternate choices of hooks,
\begin{equation*}
\Res_{u=[\bms{v}]}  T_{\mu\*\bmu}(u) =
\frac{\prod_{b \in \mu} [\bb]}{\prod_{b \in \mu^\times} [b] } 
\prod_{b \in \mu} \frac{\lh{\mu}{b} }{\lh{m^n}{\vflip{b}}} \prod_{b \in \bmu} \frac{\uh{\bmu}{b}}{\uh{m^n}{\hflip{b}}}.
\end{equation*}
We group these terms as follows,
\begin{eqnarray*}
&=& \frac{ \jacktop_{m^n}/\jacktop_{\bmu} }{ \jacktop_{\mu} }   \frac{ \prod_{b \in \mu}\lh{\mu}{b} \prod_{b \in \bmu}  \uh{\bmu}{b} }{\left( \prod_{b \in \mu} \lh{m^n}{\vflip{b}}   \prod_{b \in \bmu} \uh{m^n}{\hflip{b}} \right) }.
\end{eqnarray*}
The last thing we need to check is the following expression for the denominator
\begin{equation*}
\prod_{b \in \mu} \lh{m^n}{\vflip{b}}   \prod_{b \in \bmu} \uh{m^n}{\hflip{b}} = \prod_{b \in \mu} \lh{m^n}{\vflip{b}}   \prod_{b \in m^n/\mu} \uh{m^n}{\hflip{\bb}}  = \prod_{b \in m^n} D'_{\mu,m^n}(b),
\end{equation*}
where we have used $\hflip{\bb} = \vflip{b}$.
\end{proof}

Note, that by swapping $\mu \to \bmu$ in formula~\eqref{eq:resexpression} 
we have the two representations of this Littlewood--Richardson coefficient (see Figure~\ref{fig:rectcasetwopresentations}):
\begin{equation}\label{eqn:rectangularLRalt}
\jackjLR_{\mu,\bmu}^{m^n}  =  \frac{ \prod_{b \in \mu}\lh{\mu}{b} \prod_{b \in \bmu}  \uh{\bmu}{b} }{\prod_{b \in m^n} D'_{\mu,m^n}(b) } = \frac{ \prod_{b \in \mu}\uh{\mu}{b} \prod_{b \in \bmu}  \lh{\bmu}{b} }{\prod_{b \in m^n} D'_{\bmu,m^n}(b) }.
\end{equation}

\begin{figure}[!htb]
\centering
\includegraphics[width=0.5\textwidth]{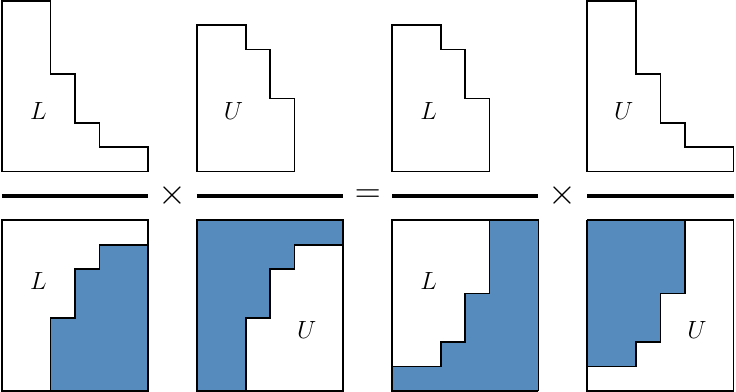}
\caption{
The two choices of hooks for the Littlewood--Richardson coefficients in the rectangular case. Each of the two fractions on either side of the equation are equal to the corresponding fraction on the other side. We color in blue the boxes that don't have hooks assigned to them.
}
\label{fig:rectcasetwopresentations}
\end{figure}

\section{Rectangular Union Case}

In this section we prove a generalization of the previous results. 
Fix $\mu$ and $m^n$, without necessarily assuming that $\mu \subset m^n$ as in the rectangular case. Let $v = (n,m)$, and $\bms{v} = (n-1,m-1)$.
Let $\sigma = \mu \cap m^n$. 
Let $\bsigma$ be the conjugation of $\sigma$ taken with respect to $m^n$, as illustrated in Figure~\ref{fig:rectunioncase}.
\begin{figure}[!htb]
\centering
\includegraphics[scale=1.0,page=5]{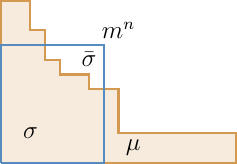}
\caption{
The setup of the rectangular union case.
}
\label{fig:rectunioncase}
\end{figure}

\begin{lemma}
We have
\begin{equation}\label{eq:claim1}
 \jackjLR_{\mu,\bsigma}^{\mu\cup m^n} =
   \jackjLR_{\sigma,\bsigma}^{ m^n} \cdot
  \prod_{b \in \mu/\sigma} T_{\bsigma} \left([\bb] \right)  .
\end{equation}
\end{lemma}
\begin{proof}
First, by Lemma~\ref{lemma:Tpoles}, we know $ T_{\mu,\bsigma}(u)$ has a pole at $u=[\bms{v}]$. By formula~\ref{eq:sumproductident}, we know that
\begin{equation*}
\sum_{\gamma} \hat  \jackjLR_{\mu,\bsigma}^\gamma \sum_{b \in \gamma} \frac{1}{u-[b]} = T_{\mu\*\bsigma}(u)-1.
\end{equation*}
The residue of this equality at $u = [\bms{v}]$ is
\begin{equation*}
\sum_{\gamma: \mu\cup \bsigma \subseteq \gamma, \bms{v} \in \gamma} \hat \jackjLR_{\mu,\bsigma}^\gamma  =\Res_{u=[\bms{v}]} T_{\mu\*\bsigma}(u).
\end{equation*}
Because $\mu\cup m^n$ is the only partition of its size that contains $\mu$ and the box at $\bms{v}$, the sum on the LHS contains only the $\gamma=\mu\cup m^n$ term and we conclude
\[ 
\hat \jackjLR_{\mu,\bsigma}^{\mu\cup m^n}   =\Res_{u=[\bms{v}]} T_{\mu\*\bsigma}(u).
\]

Next, we note that
\[ 
T_{\mu\*\bsigma}(u) =  T_{(\mu/\sigma)\*\bsigma}(u) T_{\sigma\*\bsigma}(u),
\]
and as the second factor has a pole $u=[\bms{v}]$, the first factor is regular at $u=[\bms{v}]$. So,
\[ 
\Res_{u=[\bms{v}]} T_{\mu\*\bsigma}(u) = T_{(\mu/\sigma)\*\bsigma}([\bms{v}]) \Res_{u=[\bms{v}]} T_{\sigma\*\bsigma}(u).
\]
We then use
\[
T_{\{s\} *\bsigma}([\bms{v}])=T_{\bsigma}([\bms{v}-s]),
\]
to arrive at
\begin{equation}\label{eq:withhats}
\hat \jackjLR_{\mu,\bsigma}^{\mu\cup m^n} = \left( \prod_{b \in \mu/\sigma} T_{\bsigma}([\bms{v}-b]) \right) \cdot  \hat\jackjLR_{\sigma,\bsigma}^{ m^n} .
\end{equation}
Finally, we notice that
\begin{equation*}
\frac{\varpi_\mu \varpi_{\bsigma}}{\varpi_{\mu\cup m^n}} = \frac{\varpi_\sigma \varpi_{\bsigma}}{\varpi_{m^n}},
\end{equation*}
so we can drop the hats in equation~\ref{eq:withhats}.
\end{proof}

In order to expand the product term in \eqref{eq:claim1}, we produce a formula for the factors appearing therein.

\begin{lemma}For $\mu$ as before, and $s \in \addset_{\mu}$ with $s$ outside of $m^n$, then
\begin{equation}\label{eq:tboxfactor}
T_{\bsigma}([\bar s]) = \frac{
\frac{\displaystyle\prod_{b \in R^s_{\mu \cup m^n}  } \uh{\mu \cup m^n}{b} }{
\displaystyle\prod_{b \in R^s_{(\mu+s) \cup m^n} } \uh{(\mu+s) \cup m^n}{b} } \times
\frac{\displaystyle\prod_{b \in C^s_{\mu \cup m^n}  } \lh{\mu \cup m^n}{b}}{
\displaystyle\prod_{b \in C^s_{(\mu+s) \cup m^n}  } \lh{(\mu+s) \cup m^n}{b}}}{  
\frac{\displaystyle\prod_{b \in R^s_{\mu} } \uh{\mu }{b} }{
\displaystyle\prod_{b \in R^s_{\mu+s}   } \uh{\mu+s}{b}} \times  \frac{
\displaystyle\prod_{b \in C^s_\mu  } \lh{\mu}{b}}{
\displaystyle\prod_{b \in C^s_{\mu+s}  } \lh{\mu+s}{b} }}.
\end{equation}
\end{lemma}
\begin{proof}

We start by noting that for any $\mu$ such that $\mu \cap m^n = \sigma$, we have
\[
T_{\bsigma}([n,m]-u)  = \frac{T_{\mu \cup m^n}(u)}{T_{\mu}(u)}.
\]
So
\[ 
T_{\bsigma}([\bar s])=T_{\bsigma}([(n,m)-(\bps{s})])=
\frac{T'_{(\mu+s) \cup m^n}([\bps{s}])}{T'_{(\mu+s)}([\bps{s}])},
\]
where $T'$ indicates that we have dropped the zero at $u=[\bps{s}]$.
We then use the expansion Lemma~\ref{lemma:expansion} to write the terms in this ratio as

\[ 
T'_{(\mu+s) \cup m^n}([\bps{s}]) =- [\bps{s}] \times
\frac{\prod_{b \in R^s_{\mu \cup m^n}  } \uh{\mu \cup m^n}{b} }{
\prod_{b \in R^s_{(\mu+s) \cup m^n} } \uh{(\mu+s) \cup m^n}{b} } \times
\frac{\prod_{b \in C^s_{\mu \cup m^n}  } \lh{\mu \cup m^n}{b}}{
\prod_{b \in C^s_{(\mu+s) \cup m^n}  } \lh{(\mu+s) \cup m^n}{b}}.
\]

\[ T'_{(\mu+s)}([\bps{s}]) = - [\bps{s}] \times  \frac{\prod_{b \in R^s_{\mu} } \uh{\mu }{b} }{
\prod_{b \in R^s_{\mu+s}   } \uh{\mu+s}{b}} \times  \frac{\prod_{b \in C^s_\mu  } \lh{\mu}{b}}{
\prod_{b \in C^s_{\mu+s}  } \lh{\mu+s}{b} }.
\]
Taking the ratio of these these two equalities, we arrive at the result.
\end{proof}

To state the following results, we need the following
decomposition of $\mu = (\mu_0,\mu_1,\mu_2,\mu_3)$ w.r.t a
rectangle $m^n$ as in Figure~\ref{fig:mudecomp}.
\begin{figure}[!htb]
\centering
\includegraphics[scale=1.0,page=1]{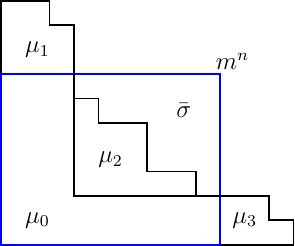}
\caption{The decomposition $\mu =: \mu_0 \cup \mu_1 \cup \mu_2 \cup \mu_3 $ corresponding to the intersection with the rectangular partition $m^n$, and the definition of $\bsigma$ as the (reverse of the) boxes that fill the remainder of the rectangle. Note that $\mu_0$ is defined as the union of all boxes under $\mu_1$ with those left of $\mu_3$.}\label{fig:mudecomp}
\end{figure}

\begin{theorem}\label{thm:rectunionjacklr}
The Jack Littlewood--Richardson coefficient
$\jackjLR_{\mu,\bsigma}^{\mu\cup m^n }$
in the rectangular union case is given by the fraction
of hooks in Figure~\ref{fig:rectunionLR}.
That is, the figure gives a prescription for upper/lower hooks for each box
in each of the three partitions.
\begin{figure}[!htb]
\centering
\includegraphics[scale=0.8,page=2]{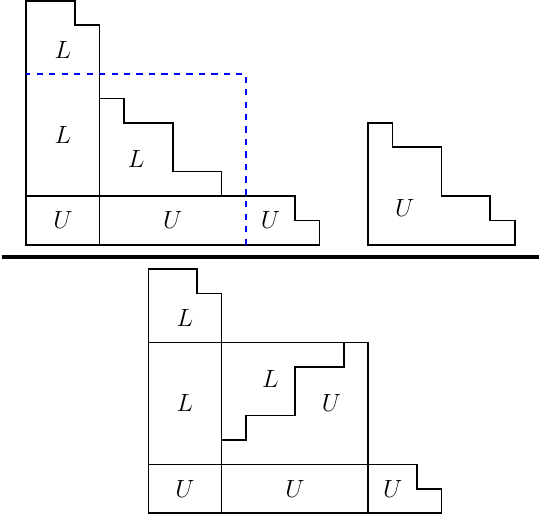}
\caption{
The coefficient $\jackjLR_{\mu,\bsigma}^{m^n \cup \mu}$ expressed as a fraction involving the assignment of upper or lower hooks for every box in each of the three partitions.
}\label{fig:rectunionLR}
\end{figure}
\end{theorem}
\begin{proof}

We refer to the decomposition in Fig.~\ref{fig:rectunionLR}, according to which the the result~\eqref{eq:claim1} becomes
\begin{equation}\label{eqn:gmusigma}
\hat \jackjLR_{\mu,\bsigma}^{\mu\cup m^n} =\left( \prod_{ b \in \mu_3} T_{\bsigma}([\bb]) \right)  \left( \prod_{b \in \mu_1} T_{\bsigma}([\bb]) \right) \cdot \hat \jackjLR_{\sigma,\bsigma}^{ m^n}.
\end{equation}
We will proceed by beginning with $\sigma= \mu \cap m^n = \mu_0 \cup \mu_2$, 
and then first extend to include $\mu_1$ by adding one box at a time, and then extending to $\mu_3$
in a similar manner.

One of these two extensions will be simpler, depending one which choice of hooks we take for $\jackjLR_{\sigma, \bsigma}^{m^n}$.

We start with by choosing the hooks as given by Equation~\ref{eqn:rectangularLR}, represented by
Fig.~\ref{fig:rectunionLRStart}, noting that the technique works similarly if we begin with the alternate choice. Importantly, with this choice, every box in $\mu_0$ that is beneath $\mu_1$ has been assigned lower hooks, in both the $\sigma$ factor in the numerator and the $\sigma \cup m^n$ factor in the denominator. This will make including the boxes of $\mu_1$ rather straightforward.

\begin{figure}[!htb]
\centering
\includegraphics[scale=0.5,page=3]{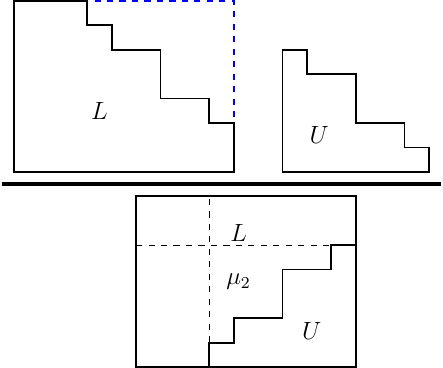}
\caption{The formula for $\jackjLR_{\sigma, \bsigma}^{m^n}$, indicating the location of the flipped component of $\mu_2$ }
\label{fig:rectunionLRStart}
\end{figure}

We now proceed to add in each of the boxes in $s \in \mu_1$.
Since $\mu_1$ is in the top-left region above $m^n$,
for $b \in \mu_1$ we have
\[ 
\ulh{\mu}{b} = \ulh{\mu\cup m^n}{b}.
\]
Because of this, and because $s \in \mu_1$ implies $R_{\mu_1}^s \subset \mu_1$,
the factor (formula~\eqref{eq:tboxfactor}) according to equation \eqref{eqn:gmusigma} coming from adding of the box $s\in\mu_1$ simplifies to
\[ 
T_{\bsigma}([\bar s])=
\frac{\prod_{b \in C^s_{\mu \cup m^n}  } \lh{\mu \cup m^n}{b} }{
\prod_{b \in C^s_{(\mu \cup s) \cup m^n}  } \lh{(\mu  \cup  s) \cup m^n}{b} } /
\frac{\prod_{b \in C^s_\mu  } \lh{\mu}{b} }{ \prod_{b \in C^s_{\mu  \cup  s}  } \lh{\mu \cup s}{b}}.
\]
We represent this `new factor' pictorially in the right hand side of Fig.~\ref{fig:Lmu1ext}. We color in orange those boxes that we are tasked with assigning hooks to, and color in blue those which aren't assigned hooks. 
Of the terms appearing in this new factor, two will remove the $L$-hooks for all $b\in C_\sigma^s$ and $b\in C_{\sigma \cup m^n}^s$, and the other two terms will add these hooks back in for the larger 
partitions $\sigma\cup s$ and $(\sigma \cup m^n) \cup s$, and thus extending the choice of a $L$-hook to the new box in the column, as seen in Fig.~\ref{fig:lowerhookmu1added}.
By continuing in this way for each box $b \in \mu_1$, each new box will always have all lower hooks beneath it, and thus these factors provide exactly for $L$-hooks for all the boxes in $\mu_1$.

\begin{figure}[!htb]
\centering
\includegraphics[scale=0.5,page=5]{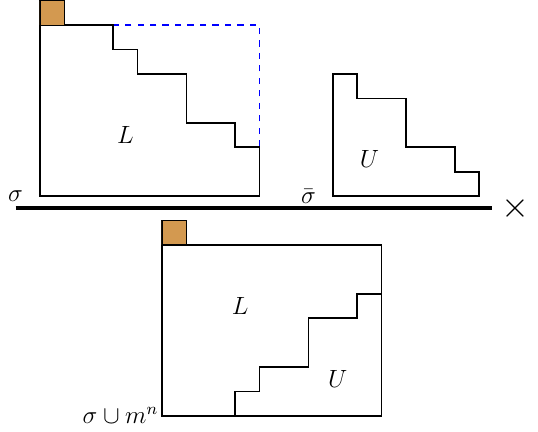}
\includegraphics[scale=0.5,page=6]{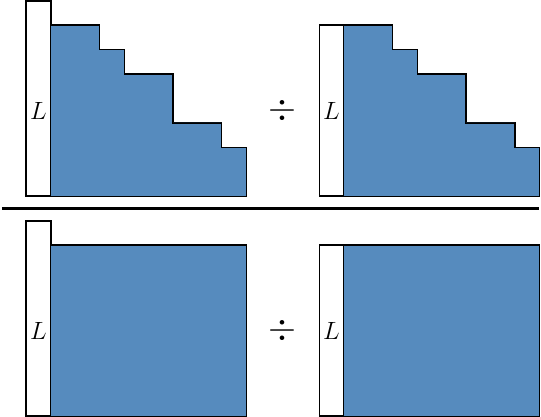}
\caption{The new orange box $s \in \mu_1$ being added, and on the left we have the factor $T_{\bsigma}([\bs])$. The task is to assign hooks to the added orange box(es), using the hooks coming from the `new factor' on the right, where blue indicates boxes that don't have assigned hooks.}
\label{fig:Lmu1ext}
\end{figure}
\begin{figure}[!htb]
\centering
\includegraphics[scale=0.5,page=7]{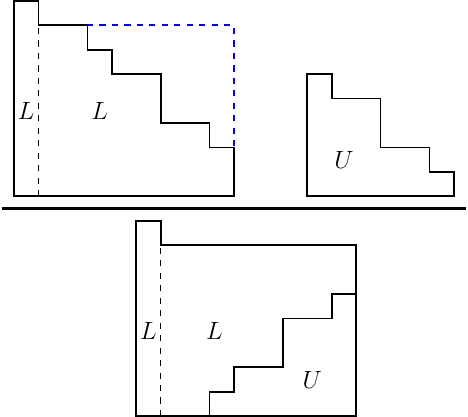}
\caption{We notice that these factors precisely extend the column of $L$-hooks on the far left up one box.}
\label{fig:lowerhookmu1added}
\end{figure}

Next, we proceed to adding the first box $s$ of $\mu_3$.
Unlike the boxes we added in $\mu_1$, $s \in \mu_3$ does not have all $U$-hooks to the right of it,
and so a more involved process will be required to add in such a box.
We observe that for $b \in \mu_3$ we have
\[ 
\ulh{\mu}{b} = \ulh{\mu\cup m^n}{b}.
\]
and thus, similar to before, we find that a box $s \in \mu_3$ implies that $C_{\mu_3}^s \in \mu_3$, and so the new factor reduces to
\[ 
T_{\bsigma}([\bar s])=
\frac{\prod_{b \in R^s_{\mu \cup m^n}  } \uh{\mu \cup m^n}{b} }{
\prod_{b \in R^s_{(\mu \cup s) \cup m^n}  } \uh{(\mu  \cup  s) \cup m^n}{b} } /
\frac{\prod_{b \in R^s_\mu  } \uh{\mu}{b} }{ \prod_{b \in R^s_{\mu  \cup  s}  } \uh{\mu \cup s}{b}}.
\]
This new factor is represented in Figure~\ref{fig:mu3firstbox}.

\begin{figure}[!htb]
\centering
\includegraphics[scale=0.5,page=8]{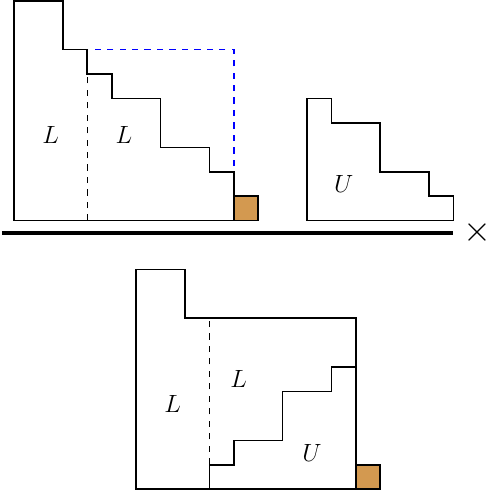}
\includegraphics[scale=0.5,page=9]{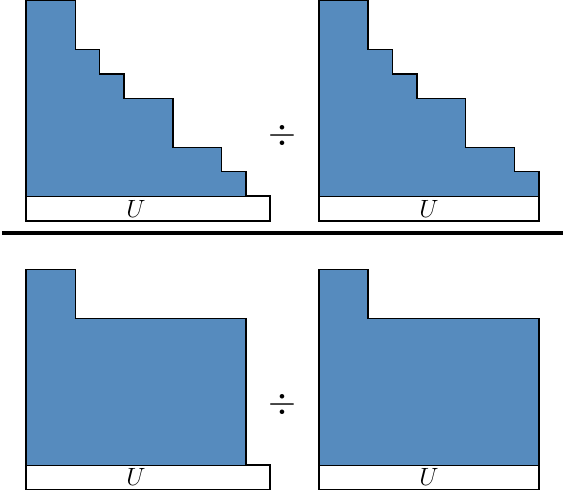}
\caption{On the left, we show the new factor $T_{\bsigma}([\bs])$ corresponding to adding the box $s \in \mu_3$ ($s$ drawn in orange).}
\label{fig:mu3firstbox}
\end{figure}
The next step is to swap the row $R_\mu^s$ of $L$-hooks in $\mu$ that contains $s \in \mu_3$, 
with the row of $U$-hooks in $\mu \cup s$ from the new factor, as indicated in Figure~\ref{fig:step1rowswap}.

\begin{figure}[!htb]
\centering
\includegraphics[scale=0.5,page=10]{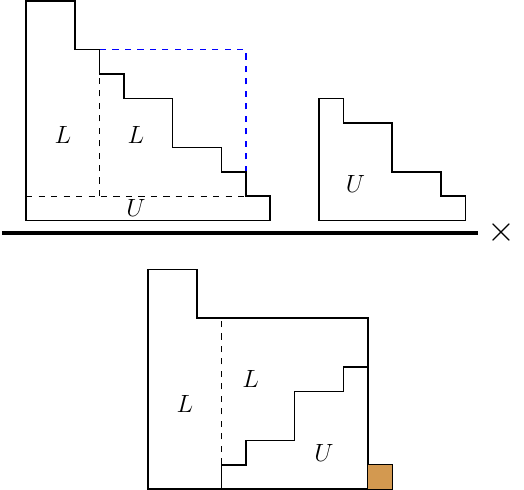}
\includegraphics[scale=0.5,page=11]{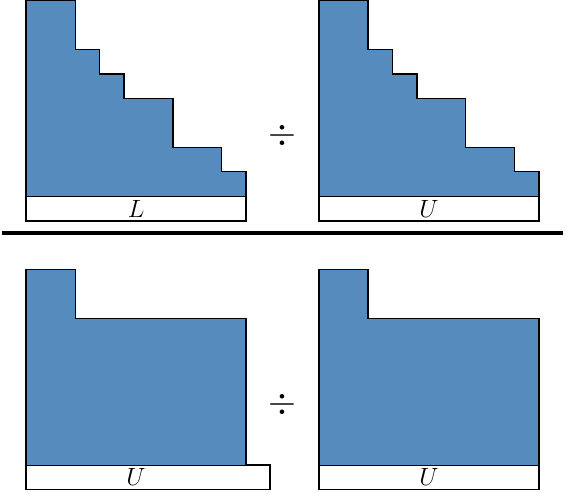}
\caption{The result of swapping the first term in the numerator of the new factor with the bottom row of the $\mu$ factor.}
\label{fig:step1rowswap}
\end{figure}

Then, we split up the two rows of hooks in the numerator of the new factor according 
to the staggered row profile of $L$-hooks in the denominator of the right hand side, as in Figure~\ref{fig:step2splitrow}.
\begin{figure}[!htb]
\centering
\includegraphics[scale=0.5,page=1]{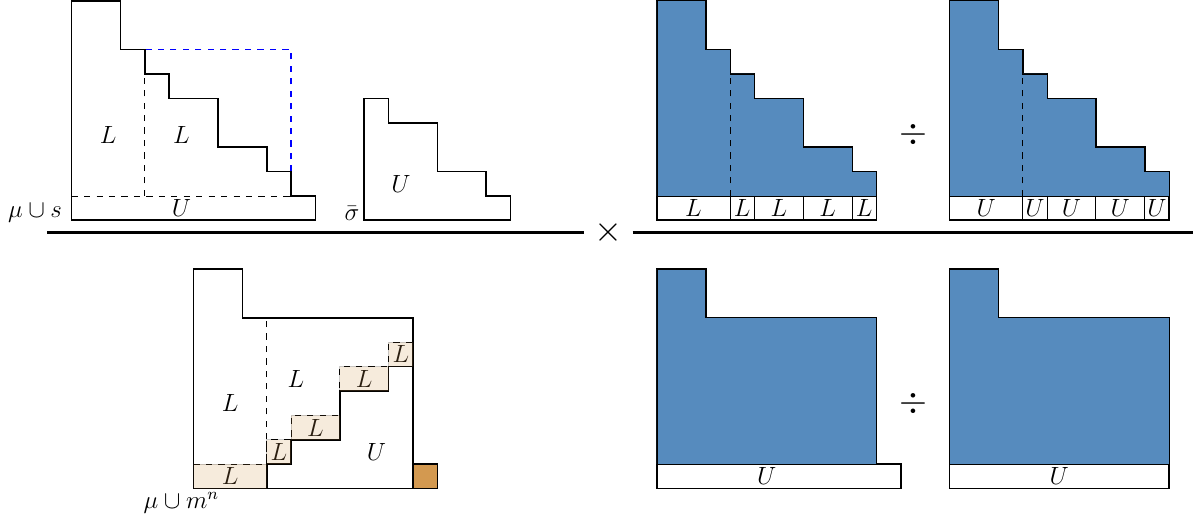}
\caption{Split up the two rows in the numerator of the new factor according to the profile of lower hooks in the denominator of the left hand side.}
\label{fig:step2splitrow}
\end{figure}
We now show that this numerator of the new factor gives precisely the hooks required to 
flip the indicated hooks in the denominator of the right hand side from $L \to U$.
Let $\ell_{\lambda}(b)$ be shorthand for $\leg_\lambda(b)$. 
Let $b \in \mu$ be such that the row $R_\mu^b$ extends all the way to 
the right boundary of $m^n$ (for example, in the case at hand $b$ is in the bottom row of $\mu$). 
The claim is
\begin{equation*}
\lh{\mu}{b} = \lh{\mu \cup m^n}{b+(0,\ell_{\mu \cup m^n}(b)-\ell_{\mu}(b))}.
\end{equation*}
One can easily verify this by direct calculation or pictorially. 
Thus, we use the hooks on the numerator of the new factor to flip all the hooks on the 
staggered row in the denominator. After this, we are left with the 
two terms in the denominator of the new factor, as in Figure~\ref{fig:step2flippedstagger}.

\begin{figure}[!htb]
\centering
\includegraphics[scale=0.5,page=2]{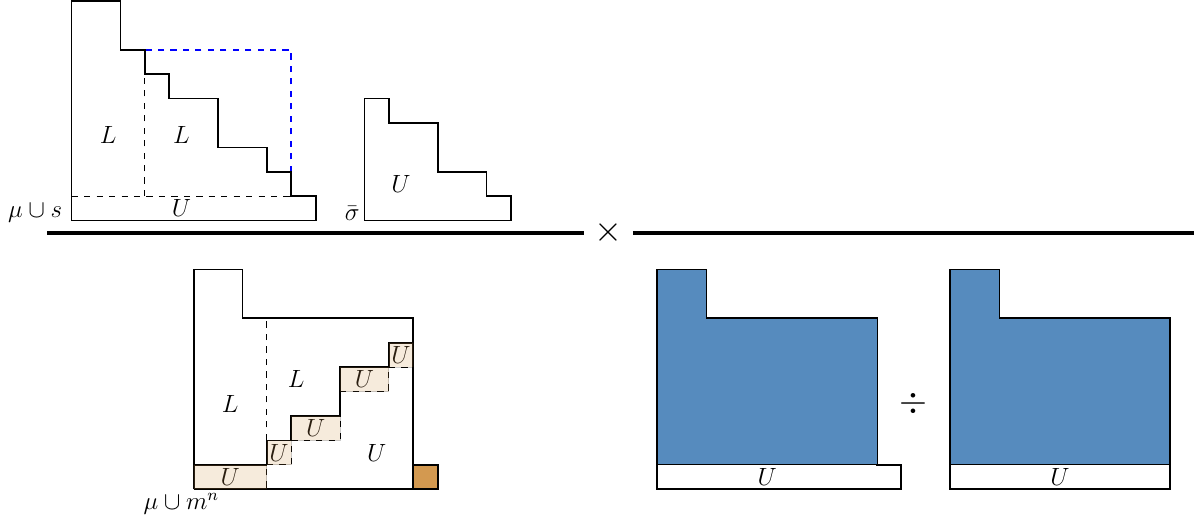}
\caption{With the staggered row now flipped to upper hooks, the bottom row of the denominator can now easily be extended.}
\label{fig:step2flippedstagger}
\end{figure}

Now that the entire bottom row of the denominator of the RHS is all upper hooks, 
we can see that these two remaining rows in the denominator of the new factor are 
precisely what is required to extend this row of upper hooks to include $s \in \mu_3$. 
With this, all hooks in the new factor have been used, and all boxes in the 
triple of partitions have been assigned a choice of hooks, as in Figure~\ref{fig:completed}.

\begin{figure}[!htb]\centering
\includegraphics[scale=0.5,page=5]{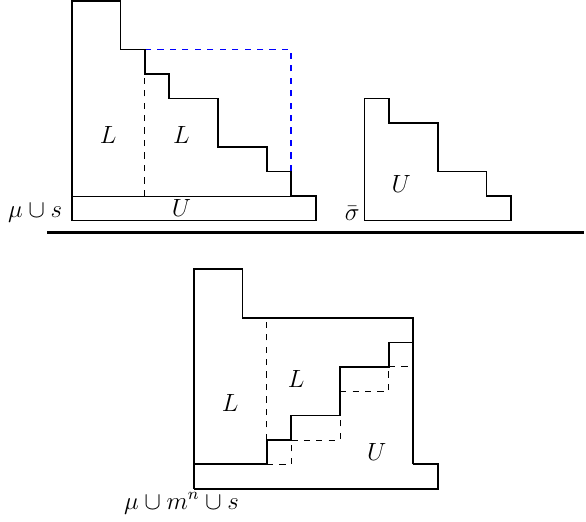}
\caption{After extending the bottom row, we have assigned hooks for all boxes including $s$.}
\label{fig:completed}
\end{figure}

Now, we can use the above steps to include the the rest of the boxes inside $\mu_3$. 
If the newly added box is to the right of a box added previously, the row just extends by $U$-hooks. 
If the new box in $\mu_3$ is added to a new row, then repeat all of the above steps, starting in this new row.
Proceeding in this way, we see that all boxes to the left of $\mu_3$ will be 
assigned U-hooks, and the remaining rectangular region will 
be assigned hooks just as in Figure~\ref{fig:rectunionLR}. 
\end{proof}

As we mentioned at the start of Section~\ref{sect:stanley}, there is a direct 
correspondence between the Littlewood--Richardson coefficients and Stanley structure coefficients. 
Thus, our Theorem~\ref{thm:rectunionjacklr} immediately gives the following main result of the paper.
\begin{corollary}\label{cor:strongStanley}
The strong Stanley conjecture holds in the rectangular union case.
\end{corollary}
We find that the Stanley structure coefficients are given by Figure~\ref{fig:rectunionLR}, 
except with bringing the denominator up to the numerator with its hooks flipped $L \leftrightarrow U$.

\begin{example}
We have
\begin{equation*}
\langle J_{42211}J_{211}, J_{43331}\rangle  
=\ytableausetup{boxsize=1.0em}
\begin{ytableau}
*(boxL) L   \\
*(boxL) L   \\
*(boxL) L & *(boxL) L \\
*(boxL) L & *(boxL) L \\
*(boxU) U & *(boxU) U & *(boxU) U & *(boxU) U \\
\end{ytableau}\quad
\begin{ytableau}
*(boxU) U  \\
*(boxU) U  \\
*(boxU) U & *(boxU) U
\end{ytableau}\quad
\begin{ytableau}
*(boxU) U \\
*(boxU) U & *(boxU) U & *(boxL) L \\
*(boxU) U & *(boxU) U & *(boxL) L \\
*(boxU) U & *(boxL) L & *(boxL) L \\
*(boxL) L & *(boxL) L & *(boxL) L & *(boxL) L \\
\end{ytableau}. \\
\end{equation*}
From the diagrams, we then get (by reading columns from left to right)
\begin{align*}
&\left([1,0][2,0][3,-1][4,-1][4,-4] \cdot
[1,0][2,0][2,-3]\cdot[0,-2]\cdot[0,-1]\right) \\
& \times \left([0,-1][1,-1][2,-2]\cdot[0,-1]\right) \\
& \times \left([0,-1][1,-3][2,-3][3,-3][5,-3] \right. \\
& \phantom{\times(}\cdot [0,-2][1,-2][3,-1][4,-2]\\
& \phantom{\times(}\cdot \left.[1,0][2,0][3,0][4,-1] \cdot [1,0] \right).
\end{align*}
Using $[x,y]=\alpha x - y$, we find that the above expression
equals
\begin{align*}
2^9\cdot 3^2\cdot\alpha^6(1+\alpha)^4(2+\alpha)(3+\alpha)^2(4+\alpha)^2
 (1+2\alpha)(1+3\alpha)(2+3\alpha)^2(5+3\alpha).
\end{align*}
\end{example}

Finally, we summarize the two possible solutions to the rectangular union Littlewood--Richardson coefficient, based on the starting choice for $g_{\sigma,\bsigma}^{m^n}$, and by noticing the cancellation of the L/U hooks between the numerator and denominator of Fig~\ref{fig:rectunionLR} in the $\mu_1$, $\mu_3$ regions and a corner rectangle.
We see that the rectangular union Littlewood--Richardson coefficient can be expressed in terms of a
rectangular Littlewood--Richardson coefficient of a sub-diagram.
\begin{corollary}
Let $\mu_2$ be as before, and note that $\mu_2 \cup \bsigma^R$ is a rectangular region $K$ of shape $k^\ell$. We then have
\begin{equation*}
g_{\mu,\bsigma}^{\mu\cup m^n} = F \times g_{\mu_2, \bar \mu_2}^{k^\ell}
\end{equation*}
where $F$ is a products of $L$-hooks for every box to the left of $K$, 
and $U$-hooks for every box below it, for both $\mu$ in the numerator 
and $\mu \cup m^n$ in the denominator. See Figure~\ref{fig:simpleformula2}.
\end{corollary}
\begin{figure}[!htb]
\centering
\includegraphics[scale=0.5,page=12]{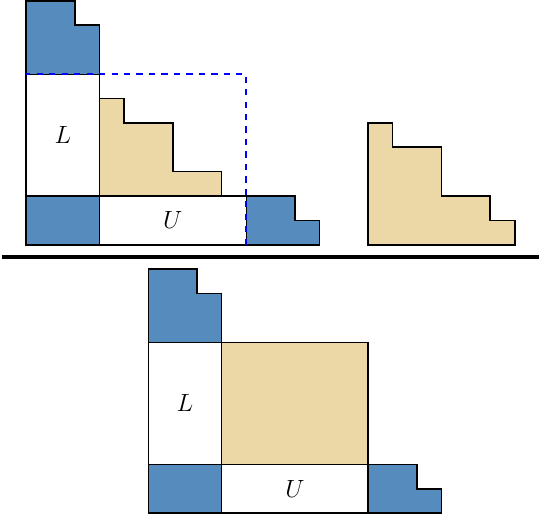}
\caption{
We can re-write the result for rectangular union in schematic form. 
The blue sections are not assigned any hooks, and for the beige regions 
we can choose one of the two solutions to the rectangular case for that subregion.}
\label{fig:simpleformula2}
\end{figure}

\section{Conjectures involving shifted Jack functions}

Is there an analog of \eqref{eq:sumproductident} for shifted Jack structure constants?
In \cite{AlexanderssonFeray:2019}, it is conjectured that for shifted Jack functions, we have that 
$\langle \jacks{\mu} \jacks{\nu} , \jacks{\lambda} \rangle$ are Laurent polynomials in $\alpha$ with non-negative
integer coefficients.
One can then ask if a shifted analog of Stanley's conjecture holds. 
Let us first set
\[
 \defin{H^U_\lambda} \coloneqq  \prod_{b \in \lambda} h^U_\lambda(b), \quad
 \defin{H^L_\lambda} \coloneqq  \prod_{b \in \lambda} h^L_\lambda(b).
\]
The definition of the Jack Littlewood--Richardson coefficients
$\defin{ \jackjLR_{\mu\nu}^{\lambda}(\alpha)}$ are then defined via
\begin{equation*}
 \jacks{\mu} \jacks{\nu} = \sum_{\lambda}  \jackjLR_{\mu\nu}^{\lambda}(\alpha) \jacks{\lambda} \text{ or equivalently }
  \jacksp{\mu} \jacksp{\nu} = \sum_{\lambda}  \jackjLR_{\mu\nu}^{\lambda}(\alpha)
  \frac{H^L_\lambda}{H^L_\mu \cdot H^L_\mu}
  \jacksp{\lambda},
\end{equation*}
where $\jacks{\mu} = H^L_\mu \jacksp{\mu}$.
When $|\mu|+|\nu|=|\lambda|$, we recover the ``non-shifted'' coefficients in \eqref{eq:jackLR},
but we may have non-zero coefficients also when $|\mu|+|\nu| > |\lambda|$.
The quantity $c_{\mu\nu}^\lambda(\alpha) \coloneqq \jackjLR_{\mu\nu}^{\lambda}(\alpha) \frac{H^L_\lambda}{H^L_\mu \cdot H^L_\mu}$
is then a deformation of the classical Littlewood--Richardson coefficients for Schur functions,
and their shifted generalization.
The inner product for the shifted Jack polynomials are then given by
\[
 \langle \jacks{\mu} \jacks{\nu} , \jacks{\lambda} \rangle =
  H^L_\lambda \cdot H^U_\lambda \cdot \jackjLR_{\mu\nu}^{\lambda}(\alpha),
\]
as in Subsection~\ref{sect:stanley}.

\begin{conjecture}[Shifted strong Stanley conjecture]
If the shifted Littlewood--Richardson coefficient $c_{\mu\nu}^\lambda(1)$ is equal to $1$,
then the (shifted) structure coefficient has the following form:
\begin{equation*}
\langle \jacks{\mu} \jacks{\nu} , \jacks{\lambda} \rangle = 
\left( \prod_{b \in \mu} \tilde{h}_{\mu}(b) \right)
\left( \prod_{b \in \nu} \tilde{h}_{\nu}(b) \right)
\left( \prod_{b \in \lambda} \tilde{h}_{\lambda}(b) \right),
\end{equation*}
where $\tilde{h}_{\sigma}(b)$ is a choice of either $\uh{\sigma}{b}$ or $\lh{\sigma}{b}$ for each box $b$.
Moreover, one chooses $\uh{\sigma}{b}$ exactly $|\lambda|$
times each in the above expression
and $\lh{\sigma}{b}$ is chosen $|\mu|+|\nu|$ times.
\end{conjecture}

\begin{example}
For $\lambda = 32211$, $\mu = 31$ and $\nu=321$, we have that 
\[
 c_{\mu\nu}^{\lambda}(\alpha) =
 \frac{\alpha ^2 (\alpha +3)^2 (\alpha +4) (2 \alpha +1)^2 (2 \alpha +5) (3 \alpha +1) (3 \alpha +2)}{(\alpha +1)^6 (\alpha +2)^2 (2 \alpha +3)^2 (3 \alpha +4)}
 \]
(which evaluates to $1$ at $\alpha=1$) and 
\[
 \langle \jacks{\mu} \jacks{\nu} , \jacks{\lambda} \rangle 
 =
 8 \alpha ^5 (\alpha +3)^2 (\alpha +4) (2 \alpha +1)^2 (2 \alpha +5) (3
   \alpha +1) (3 \alpha +2).
\]
We can see that this product is obtained from the assignments below:
\begin{equation*}
\langle  \jacks{321} \jacks{31}, \jacks{32211}\rangle
=\ytableausetup{boxsize=1.0em}
\begin{ytableau}
*(boxU) U  \\
*(boxU) U & *(boxU) U  \\
*(boxU) U & *(boxU) U &  *(boxL) L
\end{ytableau}\,\,
\begin{ytableau}
*(boxU) U \\
*(boxU) U & *(boxU) U & *(boxU) U \\
\end{ytableau}
\,\,
\begin{ytableau}
*(boxL) L   \\
*(boxL) L   \\
*(boxL) L & *(boxL) L \\
*(boxL) L & *(boxL) L \\
*(boxL) L & *(boxL) L & *(boxL) L  \\
\end{ytableau}.
\end{equation*}
Note that here the choices for all boxes that aren't inner corners are forced.
\end{example}

\section*{Acknowledgements}

The authors would like to thank the organizers of the
conference \emph{Open Problems in Algebraic Combinatorics 2022} (OPAC),
the occasion of which led to this collaboration. RM would like thank Alexander Moll for first suggesting that there might be a path to make progress on the Stanley conjectures starting from a deeper understanding of the Nazarov-Skylanin Lax operator, and this paper represents the culmination of that idea.

\bibliographystyle{amsalpha}
\bibliography{./rectangularBibliography}
\end{document}